\documentclass[a4paper,11pt]{amsart}

\usepackage{mathtools}
\usepackage{color}
\usepackage{comment}
\usepackage{amssymb}
\usepackage{amsfonts}
\usepackage{amscd}

\usepackage{slashed}
\usepackage{pdflscape}
\usepackage{tikz}
\usepackage{enumerate}
\usepackage{multirow}
\usepackage[backref,pagebackref,linkcolor=blue]{hyperref}
\usepackage{mathrsfs}
\usepackage{pifont}

\newcommand{\bm}[1]{\mbox{\boldmath~$ #1~$}}

\newtheorem{theorem}{Theorem}[section]
\newtheorem{lemma}[theorem]{Lemma}
  
\newtheorem{proposition}[theorem]{Proposition}
\newtheorem{corollary}[theorem]{Corollary}

\theoremstyle{definition}
\newtheorem{definition}[theorem]{Definition}
\newtheorem{example}[theorem]{Example}

\newtheorem{problem}[theorem]{Problem}

\theoremstyle{remark}
\newtheorem{remark}[theorem]{Remark}

\usepackage{pdfsync}

\usepackage{graphicx} 
\usepackage{amsmath} 
\usepackage{amsfonts}
\usepackage{amssymb}

\renewcommand{\O}{{\mathcal O}}

\newcommand{\Nint}{{\mathbb N}}

\newcommand{\be}{\begin{equation}}

\newcommand{\ee}{\end{equation}}

\newcommand{\sib}{\bar{\si}}

\newcommand{\cB}{{\mathcal B}}

\newcommand{\Z}{{\mathcal Z}}

\newcommand{\nablab}{\bar\nabla}

\newcommand{\cc}{\boldsymbol{c}}

\newcommand{\II}{{\bf\rm I\hspace{-.2mm}I}}
\newcommand{\IIo}{\mathring{{\bf\rm I\hspace{-.2mm} I}}{\hspace{.2mm}}}

\newcommand{\si}{\sigma}

\newcommand{\ba}{\begin{array}}

\newcommand{\ea}{\end{array}}

\newcommand{\beq}{\begin{eqnarray}}

\newcommand{\eeq}{\end{eqnarray}}

\newtheorem{lm}{lemma}

\newtheorem{thee}{theorem}

\newtheorem{proo}{proposition}

\newtheorem{co}{corollary}

\newtheorem{rem}{remark}

\newtheorem{deff}{definition}

\newcommand{\bd}{\begin{deff}}

\newcommand{\ed}{\end{deff}}

\newcommand{\bl}{\begin{lm}}

\newcommand{\el}{\end{lm}}

\newcommand{\bp}{\begin{proo}}

\newcommand{\ep}{\end{proo}}

\newcommand{\bt}{\begin{thee}}

\newcommand{\et}{\end{thee}}

\newcommand{\bc}{\begin{co}}

\newcommand{\ec}{\end{co}}

\newcommand{\brm}{\begin{rem}}

\newcommand{\erm}{\end{rem}}

\newcommand{\F}{\overline{F}}

\usepackage{t1enc}
\def\frak{\mathfrak}

\def\Cal{\mathcal}

\newcommand{\newc}{\newcommand}

\let\ccdot\cdot
\def\cdot{\hbox to 2.5pt{\hss$\ccdot$\hss}}

\newc{\aR}{\mbox{\boldmath{$ R$}}}
\newc{\aS}{\mbox{\boldmath{$ S$}}}
\newc{\aT}{\mbox{\boldmath{$ T$}}}
\newc{\aW}{\mbox{\boldmath{$ W$}}}

\newc{\aD}{\mbox{\boldmath{$ D$}}\hspace{-.2mm}}

\newc{\aK}{\mbox{\boldmath{$ K$}}}
\newc{\aL}{\mbox{\boldmath{$ L$}}}

\newcommand{\ce}{{\Cal E}}

\newcommand{\ct}{{\Cal T}}

\newcommand{\hD}{\widehat{D}}

\usepackage{amssymb}
\usepackage{amscd}

\newcommand{\nd}{\nabla}

\newcommand{\Rho}{{\rm P}}

\newcommand{\Ric}{\operatorname{Ric}}
\newcommand{\Sc}{\operatorname{Sc}}

\newcommand{\cT}{{\mathcal T}}

\newcommand{\nn}[1]{(\ref{#1})}

\newcommand{\bg}{\mbox{\boldmath{$ g$}}}

\newcommand{\J}{{\rm J}}

\newc{\obstrn}[2]{B^{#1}_{#2}}

\newcommand{\rpl}                         
{\mbox{$
\begin{picture}(12.7,8)(-.5,-1)
\put(0,0.2){$+$}
\put(4.2,2.8){\oval(8,8)[r]}
\end{picture}$}}

\newcommand{\lpl}                         
{\mbox{$
\begin{picture}(12.7,8)(-.5,-1)
\put(2,0.2){$+$}
\put(6.2,2.8){\oval(8,8)[l]}
\end{picture}$}}

\usepackage{ifthen}

\newc{\tensor}[1]{#1}
\newc{\Mvariable}[1]{\mbox{#1}}
\newc{\down}[1]{{}_{#1}}
\newc{\up}[1]{{}^{#1}}

\newc{\JulyStrut}{\rule{0mm}{6mm}}
\newc{\midtenPan}{\mbox{\sf S}}
\newc{\midten}{\mbox{\sf T}}
\newc{\midtenEi}{\mbox{\sf U}}
\newc{\ATen}{\mbox{\sf E}}
\newc{\BTen}{\mbox{\sf F}}
\newc{\CTen}{\mbox{\sf G}}

\newcommand{\w}{\mbox{\bf w}} 

\def\sideremark#1{\ifvmode\leavevmode\fi\vadjust{\vbox to0pt{\vss
 \hbox to 0pt{\hskip\hsize\hskip1em
 \vbox{\hsize2cm\tiny\raggedright\pretolerance10000
  \noindent #1\hfill}\hss}\vbox to8pt{\vfil}\vss}}}

\numberwithin{equation}{section}

\newcommand\extd{{\bm d}}

\newcommand\RR{{\mathscr R}}

\renewcommand\F{{\mathcal F}}

\newcommand{\B}{\mathcal B}

\newcommand{\Pop}{{\sf P}}

\begin{document}

\renewcommand{\today}{} \title{A Calculus for Conformal Hypersurfaces
  and new higher Willmore energy functionals} \author{A. Rod
  Gover${}^{\mathfrak G}$ \& Andrew Waldron${}^{\mathfrak W}$}

\address{${}^{\mathfrak G}$Department of Mathematics\\
  The University of Auckland\\
  Private Bag 92019\\
  Auckland 1142\\
  New Zealand} \email{r.gover@auckland.ac.nz}
  
  \address{${}^{\mathfrak W}$Department of Mathematics\\
  University of California\\
  Davis, CA95616, USA} \email{wally@math.ucdavis.edu}

\vspace{10pt}

\renewcommand{\arraystretch}{1}

\begin{abstract} 

The invariant theory for conformal hypersurfaces is studied by
treating these as the conformal infinity of a conformally compact
manifold: For a given conformal hypersurface embedding, a distinguished ambient
metric is found (within its conformal class)  by
solving a singular version of the Yamabe problem.
Using existence results for asymptotic solutions to this problem, we
develop the details of how to proliferate conformal hypersurface
invariants. In addition we show how to compute the
the solution's asymptotics. We also develop a calculus of conformal
hypersurface invariant differential operators
and in particular,
describe how to compute extrinsically coupled analogues of conformal Laplacian
powers.  Our methods also enable the study of integrated conformal
hypersurface invariants and their functional variations. As a main
application we develop new higher dimensional analogues of the Willmore
energy  for embedded surfaces. This complements recent
progress on the existence and construction of such functionals.

\vspace{2cm}


\end{abstract}

\maketitle

\pagestyle{myheadings} \markboth{Gover \& Waldron}{Conformal Hypersurface Calculus and Functionals}

\newpage

\tableofcontents
\newpage

\section{Introduction}

The data for our study is a $d$-dimensional Riemannian manifold $(M,g)$ equipped with a smoothly embedded, for simplicity oriented, codimension~1 submanifold~$\Sigma$, commonly termed a {\it hypersurface}:

\begin{center}
\includegraphics[scale=.35]{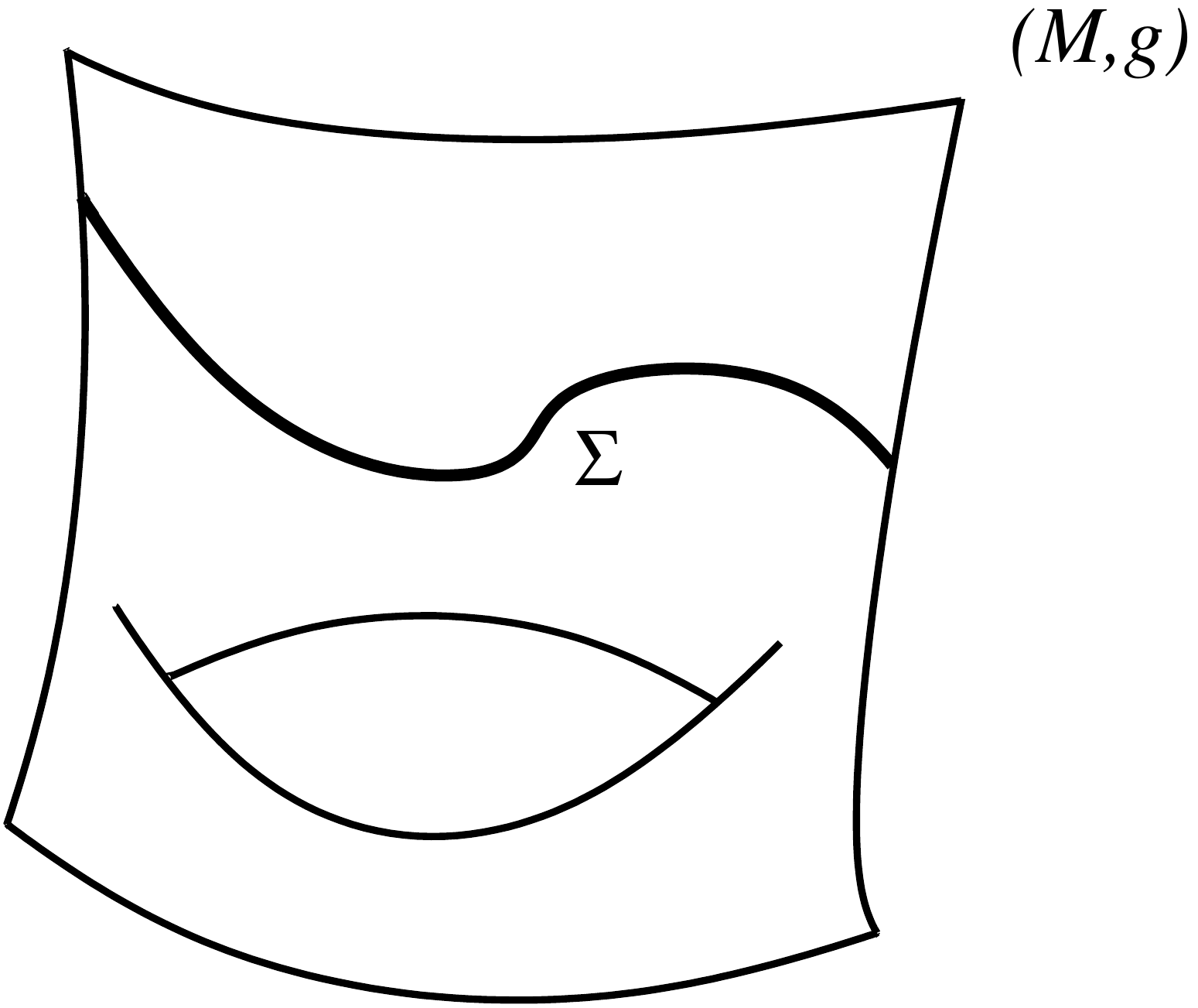}
\end{center}

\noindent
Our aim is to develop a calculus for the study of conformal
hypersurfaces including the natural invariant differential operators
associated with these and {\em conformal hypersurface invariants}.
The latter are natural density-valued tensor fields defined along~$\Sigma$ and determined by the data $(M,g,\Sigma)$, such that, as
densities, they are unchanged when $g$ is replaced by a conformally
related metric~$\Omega^2 g$ where $\Omega$ is a positive
function. 
Among such invariants there are some distinguished
invariants \cite{GW15} that, in a precise sense, provide higher
dimensional analogues of the celebrated Willmore equation studied in
{\it e.g.}~\cite{Marques,Willmore}. Recently 
energy functionals for these objects 
have been constructed from 
conformal anomalies in
a renormalised volume expansion~\cite{Gra} (see also
\cite{GW16}).
 A second main aim here is to apply tools developed in
\cite{Us,GW15} to provide a construction of manifestly  conformally
invariant energies with the same leading order functional gradient (with
respect to variation of embedding) as the anomaly functionals. 
Not only do these new
energies yield
alternative conformally
invariant higher Willmore equation, the
 nature of these suggests they
will also be useful for analysing
and even altering the
functionals in~\cite{GW16,Gra}.
Alterations may be
useful because the positivity of these higher ``energies'' is not
established.  It is also shown in \cite{GW16} that these global invariants
are related to a notion of $Q$-curvature for conformal hypersurfaces.

It is by now well-established that aspects of the {\it intrinsic}
conformal geometry of a hypersurface~$\Sigma$ can be effectively
treated by taking, at least in some collar neighborhood of~$\Sigma$,
the bulk metric $g$ to be the Poincar\'e--Einstein metric of
Fefferman--Graham (FG)~\cite{FGast}. This amounts to solving
Einstein's equations for metrics that are singular
along~$\Sigma$. Unfortunately this approach is not suitable for a
study of hypersurface invariants since it essentially forces the
embedding of $\Sigma$ to be {\em totally
  umbilic}~\cite{LeB-Heaven,Goal}, {\it i.e.}, everywhere vanishing
trace-free second fundamental form. However, in a companion
paper~\cite{GW15}, we showed that the singular Yamabe problem provides
exactly the right weakening of the Poincar\'e--Einstein condition to
yield a powerful ``holographic'' framework for the study of conformal
hypersurface invariants.

\begin{problem}[Singular Yamabe]\label{conffirststep}
Given an oriented hypersurface $\Sigma$, find a smooth function $\sigma$ such that
\begin{enumerate}[(i)]
\item $\sigma$ is a {\it defining function} for $\Sigma$ (so $\Sigma$ is the zero locus~${\mathcal Z}(\sigma)$ and $\!\extd \!\! \sigma \neq 0$ along~$\Sigma$); and
\item the singular metric $g^o=g/\sigma^2$ has scalar curvature $\Sc^{g^o}=-d(d-1)$.
\end{enumerate}
\end{problem}
\noindent
The second part of this problem is governed by the non-linear pde
\begin{equation}\label{Sgsigma}
S(g, \sigma):= |\!\extd\!\! \sigma|_g^2 - \frac 2d\, \sigma \big[\Delta^g + \frac{\Sc^g}{2(d-1)}\big] \sigma = 1\, .
\end{equation}
Here $\!\!\extd\!$ is the exterior derivative and $\Delta^g$ is the (negative energy) Laplacian.
Clearly, since the metric-defining function pair $(\Omega^2 g, \Omega \sigma)$ define the  same singular metric~$g^o$, the above equation is conformally invariant; $S(\Omega^2 g, \Omega \sigma) = S(g,\sigma)$. Therefore the above problem can be treated using conformal geometry.

\subsection{Elements of tractor calculus and the singular Yamabe problem}

A key tool for studying problems in conformal geometry is the tractor calculus 
 of~\cite{BEG} (see also \cite{GoPetCMP}). The standard
tractor bundle and its connection are  equivalent to
the normal conformal Cartan connection~\cite{CapGoirred,CapGoTAMS}, and are related to objects first
developed by Thomas~\cite{Thomas}.

Recall that 
a conformal structure~$\cc$ is an equivalence class of Riemannian
metrics where any two metrics~$g,g'\in \cc$ are related by a {\it conformal rescaling}; that is~$g'=\Omega^2 g$ with $C^\infty M\ni \Omega>0$.
Locally each~$g \in \cc$
determines a volume form and, squaring this, a section of
$(\Lambda^d T^*M)^2$. So, on a conformal manifold~$(M,\boldsymbol{c})$ there is a
canonical section~$\bg$ of~$\odot^2T^* M\otimes \ce M[2]$ called the
{\it conformal metric}.
Here~$\ce M[w]$, for any~$w\in {\mathbb R}$, denotes the
conformal density bundle.
This is the natural (oriented) line bundle 
equivalent, via the conformal structure~$\cc$, to~$\big[(\wedge^d TM)^2\big]^{\frac{w}{2d}}$.

On a conformal manifold~$(M,\cc)$, there is no distinguished connection on the tangent bundle
$TM$. However there is a canonical {\it tractor metric}~$h$ and linear connection
$\nabla^{\ct}$ (preserving~$h$; the superscript $\cT$ will often be supressed) on a related higher rank vector bundle known as
the tractor bundle~${\mathcal T}M$, which  yields a simplified treatment of  Problem~\ref{conffirststep}.
The tractor bundle~$\ct M$ is not
irreducible but has a  composition series summarised via
a semi-direct sum notation $$ \cT M= \ce M[1]\lpl
T^* M[1]\lpl\ce M[-1]\, .$$
Here $T^* M[w]:=T^*M\otimes \ce M[w]$.  A choice of metric~$g \in \cc$, or equivalently  a nowhere vanishing section $\tau$ of $\ce M[1]$ by setting $g=\tau^{-2}\bg$, determines an
isomorphism
\begin{equation*}\label{split}
\mathcal{T}M \stackrel{g}{\cong} \ce M[1]\oplus T^*\!M[1]\oplus
\ce M[-1] ~.
\end{equation*}
Computations relying on this isomorphism will be referred to as ``working in a scale'' and the section $\tau$ is called a {\it true scale} (later the term {\it scale} will be used for more  general sections of $\ce M[1]$).
We will employ an abstract index notation both for sections of tensor bundles in general and for sections $V^A$ of $\ct M$, and thus write $V^A\stackrel g= (v^+,v_a,v^-)=:[V^A]_g$ to denote the image of $V^A$ under the above isomorphism. We denote $h(V,V)$ by $V^2$, and in this scale the squared length of $V$ with respect to the tractor metric is given by \begin{equation}\label{trmet}V^2\stackrel g= 2v^+ v^- + g_{ab} v^a v^b\, .\end{equation}

It is propitious to reformulate the  notion of a defining function in terms of  densities:
A section~$\sigma$ of $\ce M[1]$ 
is said to be a {\it defining density} for a hypersurface~$\Sigma$ if~$\Sigma=\Z(\sigma)$ and~$\nabla \sigma$ is nowhere vanishing along~$\Sigma$ where $\nabla$ is the Levi-Civita connection for some, equivalently  any,  $g\in\cc$.
For a defining density~$\sigma$, we may define a corresponding {\it scale tractor}
\begin{equation*}\label{sctrac-def}
\cT M\ni I^A_\sigma
\stackrel{g}{:=} (\sigma,\nabla_a \si,-\frac{1}{d}(\Delta^g +\J)\si)=:\hD^A \sigma\, .
\end{equation*}
Here $\nabla^g$ is the Levi-Civita connection of $g$ and $\Delta^g$ its Laplacian, while $\J:=-\Sc^g/(d(d-1))$.
In  Riemannian signature, it follows immediately that
 for any defining density~$\sigma$ we have that
\begin{equation*}\label{ge0}
I_\sigma^2>0
\end{equation*}
holds in a neighbourhood of~$\Sigma$.
(We will implicitly use this fact in formul\ae\ involving the reciprocal function $1/I^2$.)
Moreover, 
$$
I^2_\sigma\stackrel g = S(g,\sigma)\, .
$$
In words, the singular Yamabe Problem~\ref{conffirststep} amounts to finding a defining density whose scale tractor has squared length equalling unity. 

It is worthwhile observing that any FG Poincar\'e--Einstein metric $g^o$ solves the singular Yamabe problem. However, for general boundary conformal geometries, the problem of finding a smooth FG Poincar\'e--Einstein metric is obstructed, and a similar statement holds for the singular Yamabe problem~\cite{ACF}. Therefore, we formulate an asymptotic version of Problem~\ref{conffirststep}:

\begin{problem}\label{I2-prob} 
Find a smooth defining density~$ \sigma$ such that
\begin{equation}\label{ind}
I^2_{\si}=1 + {\sigma}^{\ell} A_\ell\, ,
\end{equation}
for some smooth~$A_\ell\in \Gamma(\ce M[-\ell])$,
where~$\ell \in\mathbb{N}\cup\infty$ is as high as possible.
\end{problem}

Building on the foundational work~\cite{ACF}, a solution to this
problem was given in~\cite{GW15}:
\begin{theorem}
\label{BIGTHE}
\label{prod} 
Given a defining density~$\sigma_0$, there exists an improved defining
density \begin{equation}\label{expansion}{\sigma}={\hat
    \sigma}\big(1+{ \alpha_1}\hat \sigma + \cdots +{
    \alpha_{d-1}}{\hat \sigma}^{d-1}\big)\, ,\end{equation} where
$\hat \sigma= \sigma_0/{\sqrt{ I^2_{\sigma_0}}}$ in a neighborhood of
$\Sigma$, and ${ \alpha}_k$ are smooth densities, such that
\begin{equation}\label{one}
I^2_\sigma=1+{ \sigma}^d B\, .
\end{equation}
Moreover, the restriction of the weight~$w=-d$ density $B$
to the hypersurface~$\Sigma={\mathcal Z}(\sigma)$, denoted ${\mathcal B}:=B|_\Sigma$ and termed the {\it ``obstruction density''}, is  a natural 
 conformal hypersurface invariant 
which depends only on the data of the conformal embedding~$\Sigma\hookrightarrow (M,\cc)$.
\end{theorem} 

The improved defining density $\sigma$ of the theorem is unique modulo
the addition of terms of order $\sigma^{d+1}$ and any such defining
density is termed a {\it conformal unit defining density}.  Sections
of conformal (possibly tensor-valued) density bundles expressible in a
choice of scale in terms of the metric and polynomials built from jets
of $\sigma$ are termed termed {\it coupled conformal invariants}
(see~\cite[Section 6.1]{GW15} for a precise definition).  The
existence of conformal unit defining densities allows us to
proliferate conformal hypersurface invariants as encapsulated by the
following theorem:

\begin{theorem}[See \cite{GW15}]
Suppose that $\sigma$ is a conformal unit defining density and $P(\cc,\sigma)$  is a weight~$w$ coupled conformal invariant  depending pointwise on at most the $d$-jet of $\sigma$. Then the restriction of $P$ to $\Sigma$ is a weight~$w$ conformal hypersurface invariant.
\end{theorem}

Application of  this theorem requires the construction of the needed coupled
conformal invariants.  A main direction of this paper is to explain
how to systematically produce these by the application of tractor calculus.

Another main outcome of our approach is the construction of invariant differential operators determined by the conformal embedding. Notable among these are the extrinsically coupled conformal Laplacian powers $\Pop_{k}$ of~\cite{GW15}; for $k$ even these take the form~$\Delta^{k/2}$ plus lower order curvature terms and 
generalise the Laplacian powers of~\cite{GJMS}. An application of these is the construction of scalar invariants 
that cannot be directly reached from the 
above theorem. In particular, we can produce invariants of the weight $w=1-d$ 
that allows them to be integrated over a  hypersurface. This exploited in the following result: 

\begin{theorem}\label{kin}
Given a closed embedded hypersurface~$\Sigma$ in $M$, 
the functional
\begin{equation}\label{kf}
\int_\Sigma N^A \Pop_{d-1} N_A 
\end{equation}
 is a conformal invariant of~$\Sigma$.  With respect to variation of
 the embedding, the gradient of this functional is a
 conformal hypersurface invariant. For even dimensional hypersurfaces
 this is  a conformal hypersurface invariant with
 linear leading term in agreement with the obstruction density.
\end{theorem}
\noindent
Here $N_A$ is the hypersurface normal tractor of \cite{BEG},
see Equation \nn{normaltractor}.

 The last statement of 
Theorem~\ref{kin} shows that for $\Sigma$ of even dimension, these
energy functionals are genuinely quadratic at leading order, meaning
that their variational gradients are linear at linear leading order.
Thus these gradients also provide higher dimensional analogues of the
Willmore invariant and this proves that the functionals are {\em higher dimensional analogues of the Willmore energy}.  Exact agreement (not just leading order) between
the gradient and the obstruction density is verified for surfaces in
Example \ref{surfc} and for 3-dimensional hypersurfaces in
\cite{GGHW}.   Physically, these functionals (in
both dimension parities) are candidate actions for rigid membrane
dynamics.

\subsection{Structure of the article}\label{str}

Apart from the new results established here, this paper is
strongly linked to~\cite{GW15}. In one direction, an objective here is
to show how the formalism introduced in \cite{GW15} gives an effective
calculus for the computation and treatment conformal hypersurface
invariants. In the other direction, many of the results in \cite{GW15}
can only be fully appreciated and exploited when reinterpreted in
terms of basic Riemannian geometry formulae; producing these involves
considerable subtlety, and so a second objective is illustrate how
such formulae may be extracted.
  
  In Section~\ref{Riemann}, we review the
theory of Riemannian hypersurface invariants, and show how these may
be treated via a Riemannian analog of the singular Yamabe problem. In
Section~\ref{canda} we show how existence of conformal unit defining
densities alone allows us to proliferate conformal hypersurface
invariants. As an application, we compute the obstruction density in
low dimensions. Then in Section~\ref{hollow} we develop the tractor
calculus of conformal hypersurface invariants. This allows powerful
tractor techniques to be applied to these
problems. Section~\ref{invops} takes up the problem of constructing
invariant differential operators acting on conformal hypersurface
invariants. As an application, we calculate extrinsically coupled
conformal Laplacian powers in low dimensions. The final
Section~\ref{critfs} treats Theorem~\ref{kin}
and gives low dimensional examples. 

\subsection{Notation}
Our notations for standard objects in Riemannian geometry, hypersurface theory  and the conformal tractor calculus coincides with that of~\cite[Sections~2.1, 2.3 and 3.1]{GW15}, but we will also remind readers of key definitions at the appropriate junctures.

\medskip

\noindent{\bf Acknowledgements.}
Both authors would also like to thank
C.R. Graham for helpful comments.  A.W. thanks
R. Bonezzi, M. Halbasch, M. Glaros for discussions.  The authors
gratefully acknowledge support from the Royal Society of New Zealand
via Marsden Grant 13-UOA-018 and the UCMEXUS-CONACYT grant CN-12-564.
A.W. thanks  the University of Auckland for warm hospitality and
the Harvard University Center for the Fundamental Laws of
Nature. A.W. was also supported by a Simons Foundation Collaboration
Grant for Mathematicians ID 317562.

\section{Hypersurface invariants}\label{Riemann}

To prepare for our study of conformal hypersurface invariants we first demonstrate how Riemannian hypersurface invariants can be efficiently treated via an analog of the singular Yamabe Problem~\ref{I2-prob}. 
Since locally any hypersurface is the zero set of some defining function, there is no loss of generality in restricting to
hypersurfaces~$\Sigma$ which are the zero locus~$\mathcal{Z}(s)$
of some defining function~$s$. To further simplify our discussion  we
also assume that~$M$ is oriented with volume form
$\omega$. 
Given a hypersurface in~$M$, it has an orientation
determined by~$s$ and~$\omega$, as~$\boldsymbol{d} s$ is a conormal
field. Different defining functions are {\em compatibly oriented} if
they determine the same orientation on~$\Sigma$.

\begin{definition}\label{R-invtdef} 
For hypersurfaces, a {\em scalar Riemannian pre-invariant}  is a function
$P$ which assigns to each pair consisting of a Riemannian~$n$-manifold
$(M,g)$ and hypersurface defining function~$s$, a function
$P(s ; g)$ such that: \\ 
\begin{enumerate}
\item[(i)] $P(s; g)$ is natural, in the sense that
for any diffeomorphism~$\phi:M\to M$ we have~$P(\phi^* s;\phi^* g ) =
\phi^* P(s; g)$.\\ 
\item[(ii)] The restriction of~$P(s; g)$ is independent of the choice of
oriented defining functions, meaning that if~$s$ and~$s'$ are two
compatibly oriented defining functions such that
$\mathcal{Z}(s)=\mathcal{Z}(s')=:\Sigma$ then,~$P(s; g)|_\Sigma= P(s';
g)|_\Sigma$.\\
\item[(iii)] $P$ is given by a universal polynomial expression such that, 
given a local coordinate system~$(x^a)$ on~$(M,g)$,~$P(s; g)$ is given by
a polynomial in the variables
$$g_{ab},~\partial_{a_1}g_{bc}, ~\cdots, ~\partial_{a_1}\partial_{a_2}\cdots \partial_{a_k}
g_{bc},~(\det g)^{-1},$$
$$s,~\partial_{b_1} s,~\cdots~,\partial_{b_1}\partial_{b_2}\cdots \partial_{b_\ell} s, 
~|| \boldsymbol{d} s ||_g^{-1}, \omega_{a_1\ldots a_d}\, ,$$ for some positive integers~$k,\ell$. 
\end{enumerate}
A {\em scalar Riemannian invariant} of a hypersurface~$\Sigma$ is the
restriction~$P(\Sigma;g):=P(s;g)|_\Sigma$ of a 
pre-invariant~$P(s;g)$ to~$\Sigma:= \mathcal{Z}(s)$.
\end{definition}
In (iii)~$\partial_a$ means~$\partial/\partial x^a$,
$g_{ab}=g(\partial_a,\partial_b)$,~$\det g= \det(g_{ab})$ and $\omega_{a_1\ldots a_d}=\omega(\partial_{a_1},\ldots,\partial_{a_d})$.  
For
(i) note that if~$\Sigma=\mathcal{Z}(s)$, then~$\phi^{-1} (\Sigma)$
is a hypersurface with defining function~$\phi^* s$. The conditions (i),(ii) and
(iii) mean that any Riemannian invariant~$P(s;g)|_\Sigma$ of~$\Sigma$, is
entirely determined by the data~$(M,g,\Sigma)$. Then in this notation the naturality condition of (i) implies
~$\phi^*\!( P(\Sigma,g))= P(\phi^{-1}(\Sigma),\phi^* g)$.
The above
definition extends {\it mutatis mutandis}  to tensor valued
  hypersurface pre-invariants and {invariants}.
 
 \begin{example}\label{preinvariants}
 The quantities
 $$
 P(s;g)=
 \frac{1}
 {d-1}\, 
 \nabla^a\Big(
 \frac{\nabla_a s }{\, |\nabla s|}\Big)
 \mbox{ and }
 P_{ab}(s;g)=\Big(\nabla_a-\frac{(\nabla_a s)}{|\nabla s|}
\frac{(\nabla^c s)}{|\nabla s|}\, 
\nabla_c\Big)
\Big(
 \frac{\nabla_b s }{\, |\nabla s|}\Big)
$$
 are preinvariants, respectively,  for the {\it mean curvature} $H=P(s;g)\big|_{\Sigma={\mathcal Z}(s)}$
 and {\it second fundamental form}
 $\II_{ab}=P_{ab}(s;g)_{\Sigma={\mathcal Z}(s)}$.
 \end{example}

Property (ii) of  preinvariants in Definition~\ref{R-invtdef} can be exploited to expedite hypersurface invariant computations.
For example, a for many purposes simpler mean curvature preinvariant is
\begin{equation}\label{firstimprovement}
P(s_1;g) = \frac{\Delta s_1}{d-1}\, ,\mbox{ where }\, 
s_1=\frac{s}{|\nabla s|}\, \Big(1+\frac12 \, \frac{s}{|\nabla s|}\, \frac{(\nabla s).\nabla\log |\nabla s|}{|\nabla s|}\Big)\, . 
\end{equation} 
To see this, one computes $|\nabla s_1|$
and finds that 
$$
|\nabla s_1|^2 = 1 + s^2 A\, ,
$$
where the function $A$ is smooth. This implies that for {\it any} defining function $s$, we have that $s_1$ is also a defining function,
but with the improved behavior of the length of its gradient quoted above which allows the mean curvature to be computed directly from its Laplacian. Similarly, the second fundamental form preinvariant becomes simply
$$
P_{ab}(s_1;g)=\nabla_a \nabla_b s_1\, .
$$

Following this line of reasoning, we pose the following problem for Riemannian hypersurface defining functions:
\begin{problem}\label{Riemannsfirststep}
Given~$\Sigma$, a smooth hypersurface in a Riemannian manifold~$(M,g)$
find a defining function~$s$ such that~$n_{s}:=\nabla s$ obeys
\begin{equation}\label{ellp1}
|n_{}|^2=1+ s^{\ell+1} A\, ,
\end{equation}
for some~$A\in C^\infty(M)$ and~$\ell\in\Nint\cup\infty$ as high as possible.
\end{problem}

Problem~\ref{Riemannsfirststep}
can be solved by an explicit recursion to ${\mathcal O}(s^\infty)$~\cite{GW15}. 
Moreover, the recursion
uniquely determines $s$ to any given order.
Defining functions obeying $|\nabla s|=1$ 
are called {\it unit defining functions} and we shall also use this terminology in the setting where 
Problem~\ref{Riemannsfirststep} has been solved to sufficiently high order to uniquely determine the jets of $s$ required to evaluate any quantities involved.
Note that in fact $|n|^2=1$
can be solved in a neighborhood
of $\Sigma$, whereby $s$ measures the geodesic distance to the hypersurface. 
This is a standard maneuvre in the construction of Gaussian normal coordinates~(see for example \cite{Wald}).
For explicit computations the recursion is useful.

\begin{example}
Consider the hypersurface  in Euclidean space given by the graph of a smooth  function $f(x,y)$.
To compute the mean curvature 
we need data of the defining function up to its 2-jet. Thus, beginning with the defining function $z-f(x,y)$, we employ the improvement formula in Equation~\nn{firstimprovement} to find
$$
s=\frac{z-f}{\sqrt{f_x^2+f_y^2+1}}\left(
1-\frac12\, (z-f)\, 
\frac{f_x^2 f_{xx}+2f_x f_y f_{xy} +f_y^2 f_{yy}}{(f_x^2+f_y^2+1)^2}
\right)\, .
$$
It is not difficult to verify that this defining function obeys $|\nabla s|^2 =1 + s^2 A$ where $A$ is smooth. Moreover, the mean curvature is
$$
H=\frac12 \Delta s \big|_{z=f}=-\frac12 \frac{f_{xx}+f_{yy}+f_y^2f_{xx}
-2 f_x f_y f_{xy} + f_x^2f_{yy} }{(f_x^2+f_y^2+1)^{3/2}}\, .
$$
Readers will recognize the standard mean curvature formula for graphs.
\end{example}

Before developing further the calculus of unit defining functions and applying this to the singular Yamabe problem, we quickly review key ingredients of Riemannian hypersurface theory.

\subsection{Riemannian hypersurfaces}

Given a vector field $\hat n^a\in \Gamma(TM)$
such that $\hat n^a|_\Sigma$ is a unit normal, we may identify the tangent bundle $T\Sigma$ 
and the subbundle $TM^\top$ of $TM|_\Sigma$ orthogonal to $\hat n^a$. Thus we may employ this isomorphism to identify sections of $T\Sigma$ and $TM^\top$ and use the abstract indices of $TM$ to label these.
In particular the projection of tensors on $M$ to hypersurface tensors will be denoted by the symbol $\top$; for a vector $v\in\Gamma(TM)$ we thus have $v^\perp:=v-\hat n\,  \hat n .v$.

In general, objects intrinsic to $\Sigma$ will be labeled by a bar. For example, for a vector~$\bar v^a\in \Gamma(T\Sigma)$ and any extension of this to $v^a\in \Gamma(TM)$ subject to $v|_\Sigma=v^\top|_\Sigma=\bar v$, the intrinsic and ambient Levi-Civita connections, $\bar \nabla$ and $\nabla$ are related by the Gau\ss\, formula
\begin{equation}\label{hypgrad}
\bar \nabla_a \bar v^b = (\nabla^\top_a v^b + \hat n^b \II_{ac} v^c)\big|_\Sigma\, , 
\end{equation}
where the second fundamental form $\II_{ab}\in \Gamma(\odot^2 T^*\Sigma)$ is given by
\begin{equation}\label{two}
\II_{ab}= \nabla^\top_a \hat n_b\big|_\Sigma
\, .\end{equation}
Identifying $\hat n$ and $\nabla s/|\nabla s|$, we see that this formula is the origin of the preinvariant given in Example~\nn{preinvariants}.

\subsection{Unit defining functions and Riemannian hypersurface invariants}

Given a unit defining function $s$ we can proliferate Riemannian hypersurface invariants simply by computing all possible tensors built from gradients $\nabla_a\nabla_b\cdots \nabla_c s$, Riemannian invariants built from Riemann tensors, contractions of these objects and then studying their restriction to $\Sigma$. This methodology also yields efficient derivations of the relations of Gau\ss, Codazzi, Mainardi and Ricci.

For example, call $n_a:=\nabla_a s$. Then from
the second fundamental form preinvariant given in Example~\ref{preinvariants}, we see immediately that
\begin{equation}\label{IIunit}
\II_{ab}=\nabla_a n_b\big|_\Sigma\, .
\end{equation}
However,
$$\nabla_a \nabla_b n_c - \nabla_b \nabla_a n_c = R_{abcd}n^d\, .$$
Restricting the above relation to $\Sigma$, applying the projector $\top$ to the indices $a$, $b$ and~$c$, and then using the Gau\ss\ formula~\nn{hypgrad}, the above relation becomes
the well known Codazzi--Mainardi equation
\begin{equation}\label{Mainardi}
\nablab_a\II_{bc} - \nablab_b \II_{ac}= \big(R_{abcd}\hat n^d\big)^{\!\top}\, .
\end{equation}
Similar manuevres yield the Gau\ss\  equation 
\begin{equation}\label{Gauss}
\bar R_{abcd}=R^\top_{abcd}+\II_{ac}\II_{bd}-\II_{ad}\II_{bc}\, ,
\end{equation}
and Ricci relation
\begin{equation}\label{JJbar}
\II_{ab} \II^{ab} - (d-1)^2 H^2 =  \Sc -2\Ric(\hat n,\hat n)-\overline{\Sc}\, .
\end{equation}
For surfaces embedded in three dimensional Euclidean spaces, the above gives Gau\ss' {\it Theorema Egregium}.

We can also compute expressions involving higher jets of $s$:
Using the  fact~$n^an_a=1$ to all orders, it follows that~$\nabla_n n_b = n^a \nabla_a n_b
= n^a\nabla_b n_a =\frac12 \nabla_b (n^2)= 0$ to all orders along~$\Sigma$. Thus, remembering that $n^a|_\Sigma=\hat n^a$,
\begin{equation}\label{3rd}
\nabla_a\nabla_b \nabla_c {s}\big|_\Sigma= \nablab_a\II_{bc}-\hat n_b^{\phantom{2}}\II^2_{ca}-\hat n_c^{\phantom{2}}\II^2_{ab}
-\hat n_a\II^2_{bc}-\hat n_aR_b{}^d{}_c{}^e\hat n_d\hat n_e\, .
\end{equation}
Here we have denoted~$\II_a^b\II_{bc}^{\phantom{c}}=:\II^2_{ac}$.
More generally, we can compute the~$(k+1)^{\rm th}$ covariant
derivative~$\nabla_a\nabla_b\cdots \nabla_c s$ in terms of a
$\nabla^\top$ derivative of the~$k^{\rm th}$ covariant derivative
$\nabla_b\cdots \nabla_c s$ and lower transverse-order
derivatives of~${s}$ (transverse-order counts the number of transverse derivatives $\nabla_n$ in the obvious way, see~\cite{BrGoCNV} where it is called normal-order) since
$$
\nabla_a\nabla_b\cdots \nabla_c {s}= \nabla^\top_a\nabla_b\cdots \nabla_c {s} +n_a n^d \nabla_d \nabla_b\cdots \nabla_c {s}\, ,
$$ and the fact that~$n^a\nabla_b n_a = 0$ to all orders enables us to
re-express the second term in terms of~$k^{\rm th}$ derivatives of
${s}$. Thus by induction the result~\nn{3rd} generalises to compute hypersurface invariants in terms of any number of gradients of a unit defining function~$s$. We collect some useful 
identities derived from this observation in the 
following example:

\begin{example}
Expression~\nn{IIunit} for the second fundamental form implies that the mean curvature obeys
\begin{equation}\label{Hunit}
\nabla.n\big|_\Sigma=(d-1)H\, .
\end{equation}
Contracting the immediately subsequent display with
$n^a$ gives
\begin{equation}
\label{II2unit}
\nabla_n \nabla_b n_c\big|_\Sigma=
-\II_b^a \II_{ac}+R_{abcd}\hat n^a \hat n^d\, .
\end{equation}
The trace of this equation gives
\begin{equation}\label{trII2unit} \nabla_n \nabla. n\big|_\Sigma=-\II_a^b\II^a_b -\Ric(\hat n,\hat n)\, .
\end{equation}
Finally, for~$f$ any smooth extension of~$\bar f\in C^\infty \Sigma$,
the ambient and hypersurface Laplacians are related by
\begin{equation}\label{laplaceunit}
(\Delta^g-\nabla.n\,  \nabla_n -\nabla_n^2) f\big|_\Sigma=\bar\Delta \bar f\, .
\end{equation}
\end{example}

In summary, given a unit defining function, we can proliferate hypersurface invariants by  constructing ambient, coupled Weyl invariants (in the sense of Weyl's classical invariant theory). In fact, the recursion discussed above, establishes the following result:

\begin{theorem}\label{ind4R}\label{Rins}
If~${s}$ is a unit defining function for a Riemannian hypersurface
$\Sigma$ then, for any integer~$k\geq 1$, the quantity 
$
\nabla^k {s}|_\Sigma
$ may be expressed as~$\nablab^{k-2}\II$ plus a
linear combination of partial contractions involving the conormal~$\hat n$,
$\nablab^{\ell}\II$ for~$ 0\leq \ell\leq k-3$, and the
Riemannian curvature~$R$ and its covariant derivatives (to order at
most~$k-3$). 
Thus
any tensor of the form  
$$
\mbox{Partial-contraction}\big( (\nabla \cdots \nabla {s} ) \ldots (\nabla \cdots \nabla {s} )  (\nabla \cdots \nabla R ) \ldots (\nabla \cdots \nabla R ) \big) \big|_\Sigma\, ,
$$
yields a Riemannian hypersurface invariant. 
This may be re-expressed as linear combination of tensors built as partial contractions of undifferentiated conormals, as well as the second fundamental form and the Riemann curvature as well as derivatives thereof.\end{theorem}

The main thrust of our article is to treat conformal hypersurface invariants  in analogy to the construction leading to the above  theorem. 
A dictionary for this analogy is tabulated below:
\medskip
\begin{center}
\begin{tabular}{ccc}
Riemannian&&Conformal\\[1mm]\hline\\[-3mm]
unit defining function&$\longleftrightarrow$&conformal unit defining density \\[3mm]
$|\nabla s|_g^2 = 1$  
&$\longleftrightarrow$&$
|\nabla \sigma|_g^2 - \frac 2d\, \sigma \big[\Delta^g + \frac{\Sc^g}{2(d-1)}\big] \sigma = 1$\\[3mm]
Weyl's invariant theory
&$\longleftrightarrow$&
Weyl invariants via tractors 
\end{tabular}
\end{center}
\medskip
As implied by this table, a complete treatment requires that we
introduce a tractor calculus for the computation of ambient coupled
conformal invariants. However, simpler aspects of that program can
actually be handled with the elementary unit defining function
calculus described above.

\subsection{Unit defining functions and the singular Yamabe problem}\label{examples}

Theorem~\ref{prod} ensures that any defining function~$s$ can
be improved to a  defining density {\it function}~$\sigma(s)$ obeying the asymptotic singular Yamabe condition
\begin{equation}\label{mrpotatohead}
|\nabla \sigma|_g^2 - \frac 2d\, \sigma \big[\Delta^g + \frac{\Sc^g}{2(d-1)}\big] \sigma = 1+\sigma^d B_{\sigma(s)}\, ,
\end{equation} 
where $B_{\sigma(s)}$ is smooth. 
It is possible to directly implement the recursion of~\cite{GW15} to explicitly solve 
the singular Yamabe problem to 
the order required for studying the
Willmore invariant. This is very useful for applications involving explicit metrics.
 While this is technically
intensive, simplifications arise if one takes $s$ to be a unit
defining function, and in particular if one restricts to the case of a
Euclidean ambient space.
We record our solution to this problem below:
\begin{lemma}\label{improved}\label{obstructed1}
Let $s$ be a unit defining function for a hypersurface embedded in $d$-dimensional Euclidean space, and call $n=\nabla s$. Then 
solutions to Equation~\nn{mrpotatohead} 
are given  by
\begin{equation*}
\left\{
\begin{array}{lr}
\sigma(s)\stackrel{g}=s+\frac{s^2}4\, \nabla.n + \frac{s^3}{12} 
\big(\nabla_n \nabla.n +2(\nabla.n)^2\big)\, ,&d=3\, ,
\\[3mm]
\sigma(s)\stackrel{g}=s + \frac{s^2}6\, \nabla.n
+\frac{s^3}{18} 
\, (\nabla.n)^2+\frac{s^4}{144}\big(6\Delta \nabla.n+4\nabla.n\nabla_n\nabla.n
+\frac{14}{3}(\nabla.n)^3
\big)\, ,& d=4\, ,
\end{array}
\right.
\end{equation*}
with
\begin{equation}\label{S(sigma)}
\left\{
\begin{array}{lr}
B_{\sigma(s)}
=-\frac{1}{12}\, \big(2\, \Delta\nabla.n+2\, \nabla_n^2\nabla.n+8\, \nabla.n\, \nabla_n\nabla.n+3\, (\nabla.n)^3\big)\, ,
  &d=3\, ,\\[3mm]
B_{\sigma(s)}
=-\frac{1}{108}\, \Big(9\, \nabla_n\Delta\nabla.n+12\, \nabla.n\, \Delta^g\nabla.n+6\, \nabla.n\,\nabla_n^2\nabla.n 
\\[2mm] \hspace{3.6cm}+\,3\, (\nabla_i^g\nabla.n)(\nabla^i_g\nabla.n)+6(\nabla_n\nabla.n)^2
\\[1mm] \hspace{3.6cm}+\,18\, (\nabla.n)^2\, \nabla_n\nabla.n+4\, (\nabla.n)^4
     \Big)\, ,
     & d=4\, .
\end{array}
\right.
\end{equation}
\end{lemma}

\begin{proof}
The first half of this Lemma can be proved by following the algorithm given in Proposition~4.9 of~\cite{GW15} and thereafter 
computing~$S(g,\sigma)$ as given in~\nn{Sgsigma} ({\it i.e.}\ \nn{mrpotatohead}). Alternatively, since the lemma gives explicit formul\ae\ for the improved 
defining function, one can simply directly evaluate~$S(g,\sigma(s))$ for the quoted $\sigma(s)$. Either method only requires an elementary calculation.
\end{proof}

According to Theorem~\ref{BIGTHE}, the quantity $B_{\sigma(s)}$ 
yields a natural conformal hypersurface invariant upon restriction to~$\Sigma$.
Indeed, $B_{\sigma(s)}|_\Sigma$  
equals the obstruction density  computed in the scale $g$.
It is interesting therefore to compute this invariant.
For that we specialize Equations~\nn{Hunit},~\nn{trII2unit} and~\nn{II2unit} to a flat ambient space, and apply the recursion underlying  Theorem~\ref{ind4R} to find
\begin{eqnarray*}\nabla.n\big|_\Sigma&=&\II_a^a=(d-1)H\, ,\\[1mm]
\nabla_n\nabla.n\big|_\Sigma&=&-\II_{ab}\II^{ab}=-\IIo_{ab}\IIo^{ab}-(d-1)H^2\, ,\\[1mm]
\nabla_n^2 \nabla.n\big|_\Sigma&=&2\II_{ab}\II^{bc}\II_{c}^a=2\IIo_{ab}\IIo^{bc}\IIo_{c}^a
+6H\IIo_{ab}\IIo^{ab}+2(d-1)H^3\, .
\end{eqnarray*}
In the above $\IIo_{ab}$ denotes the trace-free second fundamental form
$$\IIo_{ab}:=\II_{ab}-H\bar g_{ab}\, ,$$
which is  well known to be a conformal hypersurface invariant.
It is not difficult to use these identities and
Equation~\nn{laplaceunit} to establish that
$$
\Delta \nabla.n\big|_\Sigma=(d-1)\bar\Delta H
+2\IIo_{ab}\IIo^{bc}\IIo_{c}^a
-(d-7)H\IIo_{ab}\IIo^{ab}-(d-1)(d-3)H^3\, .
$$
The above results combined with the~$d=3$ case of Lemma~\ref{obstructed1} give the following:
\begin{proposition}\label{Willmore}
For surfaces in   conformally flat three-manifolds, 
\begin{equation}\label{Wore}
B_{\sigma(0)}=-\frac13\big(\bar\Delta H + H\IIo_{ab}\IIo^{ab}\big)\, .
\end{equation}
\end{proposition}

\begin{remark}
The above result was first obtained in~\cite{ACF}.
Using the standard relation between Gau\ss\ and mean curvatures in Euclidean 3-space, namely 
${\rm K}=H^2-\frac12\IIo_{ab}\IIo^{ab}$, the above display becomes the {\it Willmore invariant}, or in other words the functional gradient of the Willmore energy functional (cf.\ \cite{GW15,Us}).
 \end{remark}

Exactly the same apparatus can be applied to the second half of Lemma~\ref{obstructed1} to give the analogous four dimensional result:

\begin{proposition}\label{4Will}
For hypersurfaces in conformally flat four-manifolds,  
$$
B_{\sigma(0)}=\frac1{6}\Big((\nablab_c\IIo_{ab})^2+
2\IIo^{ab}\bar\Delta\IIo_{ab}+\frac32\, \nablab.\IIo_a \nablab.\IIo^a
-2\bar\J\, \IIo_{ab}\IIo^{ab}  
+(\IIo_{ab}\IIo^{ab})^2
\Big)\, .
$$
\end{proposition}

An alternate proof of this proposition based on the holographic formula for the  obstruction density~$B$ given in Theorem~7.7 of~\cite{GW15} can be found in~\cite{GGHW}.

\begin{remark}
The trace-free second fundamental form~$\IIo_{ab}$, being conformally invariant, can be extended to an invariant hypersurface tractor~$L^{AB}\in\Gamma(\cT^{(AB)}\Sigma[-1])$ known as the tractor second fundamental form~\cite{Grant,Stafford}--see Section~\ref{hollow} for details. In these terms, the above display becomes 
$${\mathcal B}=\frac16\big((\bar D_A L_{BC}) (\bar D^{A} L^{BC})+(L_{AB}L^{AB})^2\big)\, ,$$
where $\bar D^A$ is the Thomas D-operator intrinsic to~$\Sigma$.
The above result provides  an independent check of conformal invariance, because this quantity is by construction a boundary conformal invariant.
\end{remark}

\section{Conformal hypersurface invariants}\label{canda}

Conformal hypersurface invariants are defined to be the Riemannian
invariants (see Definition~\ref{R-invtdef}) that are distinguished by
the property of possessing suitable covariance property under local
metric rescalings:

\begin{definition}\label{chi-def}
A {\em weight~$w$ conformal covariant} of a hypersurface~$\Sigma$ is a
Riemannian hypersurface invariant~$P(\Sigma,g)$ with the property that
$P(\Sigma,\Omega^2 g)= \Omega^w P(\Sigma,g)$, for any smooth positive
function~$\Omega$.
Any such covariant determines an invariant
section of~$\ce \Sigma[w]$ that we
shall denote~$P(\Sigma; \bg)$, where~$\bg$ is the conformal
metric of the conformal manifold~$(M,[g])$. We shall say that~$P(\Sigma; \bg)$ is a {\em
  conformal invariant} of~$\Sigma$. When $\Sigma$ is understood
  by context, the term {\em conformal hypersurface  invariant}
will refer to densities or weighted tensor fields which arise this way.
\end{definition}


\begin{example}
Given a defining function~$s$ and $g\in \cc$, the quantity
$$
P_a(s;g) = \frac{\nabla s}{|\nabla s|_g}
$$
is a preinvariant for the Riemannian hypersurface invariant $P_a(\Sigma;g)=\hat n_a$, termed the {\it unit conormal}. Since
$$P_a(s;\Omega^2 g)=\Omega P_a(s;g)\, ,$$
the unit conormal $\hat n_a$ is a weight $w=1$ conformal hypersurface invariant.
In contrast the mean curvature preinvariant $P(s;g)$ of Example~\ref{preinvariants} obeys
$$
P(\Sigma;\Omega^2 g)=\Omega^{-1}\Big(P(\Sigma;g)-\frac{\hat n .\Upsilon}{d-1}\Big)\, ,
$$
where $\Upsilon_a:=\nabla_a\log \Omega$, 
so the mean curvature is {\it not} a conformal hypersurface invariant. Note however, under metric rescalings $\Omega$ subject to $\hat n.\Upsilon=0$, {\it i.e.}, precisely those corresponding to the intrinsic conformal class of 
metrics $\bar\cc$ along $\Sigma$, the mean curvature transforms as a section of $\ce \Sigma[-1]$.
\end{example} 

\subsection{Computing the obstruction density}\label{comp-ext}

Theorem~\ref{ind4R} describes how to relate the jets of the
Riemannian canonical unit defining function~${s}$ to the regular
invariants of the Riemannian hypersurface that it defines. Here we
explain the corresponding algorithm for computing the jets of the canonical 
 conformal unit defining density~${\si}$ described in
Theorem~\ref{BIGTHE} and then apply this to computations of the obstruction density. This uses  ideas similar to the Riemannian  case, but the recursion is more subtle.

In a conformal manifold~$(M^d,\cc)$,~$d\geq 3$, we consider a
hypersurface~$\Sigma$ given as the zero locus of a smooth defining
density.  We need some key identities. For these we calculate with respect to
some metric in the conformal class,~$g \in \cc$ (but use the conformal
metric~$\bg$ to raise and lower indices).  
Recall that  a conformal unit defining density~$\sigma$ is a defining density satisfying 
\begin{equation}\label{nid}
n^2= 1-2\rho \si + \si^d B  \quad \Leftrightarrow \quad I_{\sigma}^2=1+\si^d B,
\end{equation}
for some smooth $B$, where
~$n$ is used to denote~$\nabla {\si}$, and ({\it cf.}~\nn{sctrac-def})
 \begin{equation}\label{Iform}
[I^A_{\si}]_g:=[\hD^A \si]_g = \left( \begin{array}{c} \si\\
n_a\\
\rho \end{array}\right)\, ,\qquad \rho:=\rho(\si)=-\frac{1}{d}(\Delta \si + \J \si)\, .
\end{equation}
Such a defining density exists  by  Theorem
\ref{BIGTHE} and is 
 canonical to~$\O(\si^{d+1})$. Note that in the above, $n^2$ is
defined via the conformal metric~$n^2=\bg^{-1}(\nabla \si, \nabla
\si)$. Display~\nn{nid} gives the failure of~$n$ to be a unit vector field away from~$\Sigma$.
Also, as above, we identify~$T\Sigma\cong TM^\top$ and
 shall write~$\gamma_{ab}:=g_{ab}-n_an_b$; along
$\Sigma$ this restricts to~$\bar{g}$, the induced metric. 
We will often denote the scale tractor~$I_{ \si}$ of the conformal unit defining density~$\si$ simply by~$ I$. We also heavily employ the (slightly ambiguous) notation $\stackrel\Sigma=$ to indicate equality along the hypersurface~$\Sigma$. In many instances, one side of such an equation will involve the restriction of an ambient quantity to $\Sigma$ while the other is a quantity only defined along $\Sigma$.
 As a first step, we  identify the second fundamental form in terms of the above data:

\begin{lemma}\label{HII}
\begin{equation}\label{rhoH}\rho\stackrel{\Sigma}=-H\, ,\end{equation}
and 
\begin{equation}\label{IIc}
\nabla_a n_b+\rho n_a n_b\stackrel\Sigma= \II_{ab}\, , \quad \mbox{ equivalently } \quad \nabla_a n_b\stackrel\Sigma= \II_{ab}+ H n_a n_b\, .
\end{equation}
Moreover
\begin{equation}\label{normaln}
\nabla_n n_a\stackrel\Sigma=Hn_a\, .
\end{equation}
\end{lemma}
\begin{proof}
The proof of the first statement is not essentially different than
that of~\cite[Proposition 3.5]{Goal}, which treats the
case of a conformal unit defining density subject to $I^2_{\sigma}=1$ exactly. 
For the second statement we
compute directly, along~$\Sigma$, beginning with the definition of the
second fundamental form. Since~$n$ has unit length along~$\Sigma$, we have
\begin{eqnarray*}
\II_{ab}&:=&\nabla_a^\top  n_b\\
&=&(\nabla_a-n_a \nabla_n) n_b\\
&=&\nabla_a n_b+\rho n_a n_b  \, .
\end{eqnarray*}
To reach the third line we used~\nn{nid} as follows
\begin{equation}\label{normaln0}
\nabla_n n_b=n^c\nabla_{c}n_b=n^c\nabla_{b}n_c = \frac{1}{2}\nabla_{b} n^2 \stackrel{\Sigma}{=} -\rho n_b\,  .
\end{equation}
This last result also gives Equation~\nn{normaln}.
\end{proof}

\begin{remark}\label{namingmethod}
In fact the Lemma holds when $I^2_{\sigma}=1+{\mathcal O}(\sigma^2)$.
Moreover, this Lemma is the main ingredient needed to  recover the result of~\cite{Goal} that the scale tractor for singular Yamabe structures agrees, along~$\Sigma$, with the normal tractor.
\end{remark}

The algorithm for computing the jets of~$\si$, and then the obstruction density~$\B$, now
proceeds recursively using two key results.
The first of these is
a conformal analogue of Proposition~\ref{ind4R}:
\begin{lemma}\label{ind4c}
Suppose that~${\si}$ is a conformal unit defining density for a
hypersurface~$\Sigma$  in a conformal manifold~$(M^d,\cc)$, with~$d\geq 3$. If~$g \in \cc$ and ~$k\leq d$  is a positive integer, then the quantity
$$
\nabla^k {\si}|_\Sigma ,\quad \mbox{where} \quad \nabla=\mbox{\rm Levi-Civita of } g,
$$ may be expressed as~$\nablab^{k-2}\II$ plus a
linear combination of terms where each term is a homogeneous
polynomial in various derivatives~$\nabla_n^m \rho$, with~$0\leq
m \leq k-2$, times a partial contraction involving the conormal~$n$,
$\nablab^{\ell}\II$ for~$ 0\leq \ell\leq k-3$, and the
Riemannian curvature~$R$ and its covariant derivatives (to order at
most~$k-3$).
\end{lemma}

\begin{proof}
We have~$n_a=\nabla_a \si$, and (according to~\nn{IIc})~$\nabla_a\nabla _b\si
\stackrel{\Sigma}{=} \II_{ab}+H n_an_b$.  The argument is now
completed by an induction following exactly the same logic as the
proof of Proposition~\ref{ind4R}. Formally the only new features are
that rather than~$n^2=1$, we now have~\nn{nid}, {\it i.e.},~$n^2=1-2\rho
\si+ \si^d B$, and instead of~$\nabla_a n_b|_\Sigma=\II_{ab}$, two derivatives of $\sigma$
are now governed by Equation~\nn{IIc}.
The second of these is a trivial adjustment to
substitutions, since it just affects the final evaluation along
$\Sigma$. The first means that arguments that previously used~$\nabla
n^2=0$ now incur nonzero terms. These new terms vanish along~$\Sigma$, but are picked up by transverse derivatives.
In particular, we have~$\nabla_n n^2 = -2\si\nabla_n \rho -
2\rho n^2 + \O(\si^{d-1})$. By counting, we see that we encounter~$\nabla_n^\ell n^2$ for~$\ell$ at  most~$k-1$ ($k\leq d$), so the~$\O(\si^{d-1})$ contribution never plays a {\it r\^ole}. Similarly, because of the coefficient 
$\sib$ adjacent to~$\rho$ in the Formula~\nn{nid} for~$n^2$, it follows that~$\nabla_n^{k-2}$ is the  highest ~$\nabla_n$ derivative of 
$\rho$ that is needed for the expression along~$\Sigma$. 
\end{proof}

The task of
computing~$\nabla^k {\si}|_\Sigma~$
in terms of familiar curvature quantities is not yet complete because  derivatives
$\nabla_n^\ell \rho$ remain. These are dealt with as follows. 
\begin{proposition}\label{rho-engine} Let $\sigma$ be a conformal unit defining function.
Then,
for integers~$2\leq k \leq d$,
\begin{equation}\label{rho-step}
\begin{split}
\frac{1}{2} \nabla_n^k I^2_{\sigma} + (d-k) \nabla_n^{k-1} \rho  
\stackrel{\Sigma}{=}
-\nabla_n^{k-1}\big( \gamma^{ab}\nabla_an_b \big) & - 
(k-1) \big[\nabla_n^{k-2}\big( \J+2\rho^2 \big)+(k-2)\rho\nabla_n^{k-2}\rho\big]\\[1mm]& +  \makebox{\rm LTOTs}\, ,
\end{split}
\end{equation}
where~$\makebox{\rm LTOTs}$ indicates additional terms involving lower transverse-order derivatives of~$\si$. 
In particular, for  ~$2\leq k \leq d-1$ we have
\begin{equation}\label{builder}
\begin{split}
\nabla_n^{k-1} \rho  
\stackrel{\Sigma}{=}
-\frac{1}{d-k}\Big(\nabla_n^{k-1}\big( \gamma^{ab}\nabla_a n_b \big) &+ 
(k-1) \big[\nabla_n^{k-2}\big( \J+2\rho^2 \big) +(k-2)\rho\nabla_n^{k-2}\rho\big] \Big)\\&+ \makebox{\rm LTOTs}\, ,\end{split}
\end{equation}
while 
\begin{equation}\label{Bform}
\cB \stackrel{\Sigma}{=} -\frac{2}{d!}\Big(
\nabla_n^{d-1}\big( \gamma^{ab}\nabla_a n_b \big) + 
(d-1)\big[ \nabla_n^{d-2}\big( \J+2\rho^2 \big)
+(d-2)\rho\nabla_n^{d-2}\rho\big] 
 \Big) + \makebox{\rm LTOTs}\, .
\end{equation}
\end{proposition}

\begin{proof}
Recall that from~\nn{Iform} and~\nn{trmet} we have~$ I^2_{ \sigma}=n^2+2\rho \si$. 
Thus 
\begin{eqnarray*}
\frac12 \nabla_n  I^2_{\sigma}&=& n^an^b \nabla_a n_b + \rho \, n^2 +\si
\nabla_n \rho \\ & = & -\gamma^{ab} \nabla_a n_b + \nabla^a n_a +
\rho \, n^2 +\si \nabla_n \rho.
\end{eqnarray*}
Now by the definition of~$\rho$ in~\nn{Iform} we have
~$\nabla^an_a= -d\rho -\J \si$. Using this, and once again that~$n^2=1-2\rho\si+\si^d B$, we have
\begin{equation}\label{stepk}
\frac12 \nabla_n  I^2_{\sigma} +(d-1)\rho -\si \nabla_n\rho = -\gamma^{ab} \nabla_an_b - \si (
\J+2\rho^2 - \si^{d-1}\rho B) .
\end{equation}
For~$f$ a conformal density on~$M$, and~$k\geq 1$ an integer, we have 
$$ \nabla^{k-1}_n (\si f) = \si \nabla_n^{k-1} f+
(k-1)n^2\nabla_n^{k-2} f
+\scalebox{.99}{$\frac{(k-1)(k-2)}{2}$}\, (\nabla_n n^2) \nabla_n^{k-3} f+
\cdots +(\nabla_n^{k-2} n^2) f.$$ 
Applying~$\nabla^{k-1}_n$ to both sides of Expression~\nn{stepk}, 
using the last display, and evaluating along~$\Sigma$ gives~\nn{rho-step}.
\end{proof}

Note that the last statement of the above proposition is just the~$k=d$
specialisation of~\nn{rho-step}, using also~\nn{nid}. 
Also, the Formula~\nn{rho-step} extends nicely to the case $k=1$ by the 
 first part of  Lemma~\ref{HII}.
The
right-hand-sides of the above three formul\ae\  involve at most a~$k^{\rm th}$ transverse derivative
of~${\si}$, all of which can be computed using Lemma~\ref{ind4c}
(for~$k\leq d$) except for the~$\nabla_n^\ell \rho$ terms appearing explicitly and those
produced via Lemma~\ref{ind4c}. In any case these involve~$\ell$ satisfying~$\ell \leq k-2$. So, recursively, we have a
computational algorithm which yields the following result.
\begin{theorem}\label{main-calc}
Suppose that~${\si}$ is a conformal unit defining density for a
hypersurface~$\Sigma$ in a conformal manifold~$(M^d,\cc)$, with~$d\geq 3$. If~$g \in \cc$ and ~$k\leq d$ is a positive integer, then the quantity
$$
\nabla^k {\si}|_\Sigma ,\quad \mbox{where} \quad \nabla=\mbox{\rm Levi-Civita of } g,
$$ may be expressed as~$\nablab^{k-2}\II$ plus a
linear combination of terms where each term is a homogeneous
polynomial in various derivatives~$0\leq
m \leq k-2$, times a partial contraction involving the conormal~$n$,
$\nablab^{\ell}\II$ for~$ 0\leq \ell\leq k-3$, and the
Riemannian curvature~$R$ and its covariant derivatives (to order at
most~$k-3$).

We thus obtain a formula for the obstruction density~$\cB$ in terms of the undifferentiated conormal and
the other quantities listed above. 
\end{theorem} 
\noindent The last statement may be viewed as following from
\nn{Bform} of Proposition~\ref{rho-engine}, by using the first part of
the Theorem to treat the right-hand-side thereof.

\subsubsection{Examples}

By applying the algorithm above Theorem~\ref{main-calc}, an explicit
computation of the obstruction density in any given dimension is
achieved by (i)~computing in detail the lower transverse order terms
(LTOTs) in the expression~\nn{rho-step}, (ii)~evaluating normal
derivatives of~$\gamma^{ab}\nabla_a n_b$ and (iii)~collecting terms
involving normal derivatives of ambient curvatures. The terms
in~\nn{rho-step} involving $(k-2)$ normal derivatives of~$\rho$ are
determined by previous recursions. The~$k=1$ step was encapsulated in
Lemma~\ref{HII}. The case~$k=2$, corresponding to dimension $d=2$ is 
special. We have not treated a tractor calculus when~$d=2$
as this requires additional structure. Nevertheless the ASC problem
does make sense because we may define~$I_{\sigma}^{2}$ to be the
conformally invariant quantity given by $(\nabla \sigma)^2 -\sigma
\Delta \sigma-\sigma^2 \J $ in a choice of scale. The existence
of a conformal unit defining density $\sigma$ satisfying $
I^{2}_\sigma= 1+\sigma^2 B$ can be readily verified by explicit
computations along the lines of Lemma~\ref{obstructed1}. In the following Lemma we interpret $(d-2)\Rho(n,n)$ as zero in dimension $d=2$:

\begin{lemma}\label{lineWillmore}
Let $\sigma$ be a conformal unit defining density. Then, if~$d\geq 2$,
$$
\frac 12\, \nabla_n^2  I^{2}_\sigma + (d-2) \nabla_n \rho\stackrel\Sigma=\IIo_{ab}\IIo^{ab}+(d-2)\Rho(n,n)\, .
$$
In particular, for~$d=2$ we have
$$
{\mathcal B}=0\, ,
$$
and for~$d>2$
\begin{equation}\label{nablanrho}
\nabla_n\rho\stackrel\Sigma=\frac{\IIo_{ab}\IIo^{ab}}{d-2}+\Rho(n,n)\, .
\end{equation}
\end{lemma}

\begin{proof}
Computing one normal derivative of Equation~\nn{stepk} and evaluating the result
along~$\Sigma$ using~$\nabla_n \sigma = n^2 \stackrel\Sigma = 1$, shows that lower order transverse derivative terms in~\nn{rho-step}
are absent when~$k=2$, so that
\begin{eqnarray*}
\frac12 \, \nabla_n^2 I^{\, 2} +(d-2)\nabla_n \rho \stackrel\Sigma= -\nabla_n \big(\gamma^{ab} \nabla_a n_b\big) -  \J -2 \rho^2  \, .
\end{eqnarray*}
From the previous~$k=1$ step (namely Lemma~\ref{HII}) the last term in the above display can be replaced by~$-2H^2$, so it only remains to compute the first normal derivative term on the right hand side:
\begin{equation*}
\begin{split}
\nabla_n\big(\gamma^{ab} \nabla_a n_b\big)&=-2  (\nabla_n n^b)(\nabla_n n_b)+\gamma^{ab}\big( \nabla_a \nabla_n n_b
+ R_{cabd}n^c n^d - (\nabla_a n^c) (\nabla_c n_b)\big)\\
&\stackrel\Sigma=-2H^2+\frac12\, \gamma^{ab}\nabla_a\nabla_b n^2-\Ric(n,n)-\II_{ab}\II^{ab}\\
&=  (d-3)\, H^2 -\Ric(n,n)-\II_{ab}\II^{ab}\, .
\end{split}
\end{equation*}
Here the 
second line relied on Lemma~\ref{HII} and Equation~\nn{normaln0} of its proof, while the 
third employed the fact that the operator~$\gamma^{ab}\nabla_a$ is tangential along~$\Sigma$.
Thus
\begin{equation*}
\begin{split}
\frac 12\, \nabla_n^2 I^2 + (d-2) \nabla_n \rho\stackrel\Sigma=
\II_{ab}\II^{ab}-(d-1)\, H^2+\Ric(n,n)-\J
=\IIo_{ab}\IIo^{ab}+\Ric(n,n)-\J\, .
\end{split}
\end{equation*}
When~$d=2$,~$\Ric=g\J$ and the second fundamental form has no trace-free part, so the obstruction density vanishes as claimed. For~$d>2$, the definition of the Schouten tensor implies that
$\Ric(n,n)\stackrel\Sigma=(d-2)\Rho(n,n)+\J$, which completes the proof.
\end{proof}

Recall that a conformal manifold equipped with a parallel standard tractor $I\neq 0$ is said to be {\em almost Einstein} (AE). In this case~$I$ is a scale tractor,~$\widehat D \sigma$, for some scale~$\sigma$. If~$I^2>0$,  any zero locus~$\Sigma$ of $\sigma$ is a totally umbilic, smoothly embedded hypersurface~\cite{Goal}.
  This  also follows from a  corollary of the above lemma and Lemma~\ref{HII}:
\begin{corollary}\label{gradI}
Let $\sigma$ be a conformal unit defining density for a smoothly embedded hypersurface~$\Sigma$. If~$d\geq3$, 
$$
[\nabla_a I_{\sigma}^B]\stackrel\Sigma=\begin{pmatrix}0\, \\ \IIo_a^b \\ -\frac{1}{d-2}\big[\bar \nabla.\IIo_a-n_a \IIo_{bc}\IIo^{bc}\big] \end{pmatrix}\, .
$$
\end{corollary}
\begin{proof}
This result follows by directly computing the tractor-coupled gradient of the scale tractor
$$
[\nabla_a I^B_{\sigma}] = \begin{pmatrix}
0\\\nabla_a n^b+\Rho_a^{b}\si+\rho \delta^b_a \\\nabla_a \rho - \Rho_a^b n_b 
\end{pmatrix}\, ,
$$
and evaluating this along~$\Sigma$.
The result for the middle slot requires only Lemma~\ref{HII}. For the bottom slot, one rewrites
$\nabla_a \rho= \nabla_a^\top \rho + n_a \nabla_n \rho$ the first term of which gives the gradient of mean curvature.
Then one uses the following identity obtained from the trace of the 
  Codazzi-Mainardi Equation~\nn{Mainardi}
 (valid in~$d\geq 3$)
\begin{equation}
\label{Mainarditrace}
\nablab.\IIo_a-(d-2)\nablab_a H=(d-2)(\Rho_{ab}\hat n^b)^\top\, ,
\end{equation}
to obtain
 the divergence of the trace-free second fundamental form (up to ambient curvatures). Treating the normal derivative of~$\rho$ term then requires Lemma~\ref{lineWillmore} and yields the result stated.
\end{proof}
The total umbilicity statement mentioned above the corollary follows by observing
that the parallel condition implies $I^2$ is constant. Our final example is a computation of the obstruction density for surfaces in three dimensions. First we state the main lemma.

\begin{lemma}\label{W3}
Let $\sigma$ be a conformal unit defining density. Then, if~$d\geq3$,
\begin{equation}
\begin{split}
\frac 12\, \nabla_n^3  I_\sigma^2 &+ (d-3) \nabla_n^2 \rho\stackrel\Sigma=
-\frac{1}{d-2}\Big(\nablab.\nablab.\IIo +(d-2)\IIo^{ab}[H\IIo_{ab}+\Rho_{ab}^\top]\Big)\\[2mm]
&-2\IIo^{ab}\IIo_{ac}\IIo^c_b+2\IIo^{ab}W_{cabd}n^cn^d+\frac{d-3}{d-2}\nabla_n G(n,n) +(d-3)(\nabla_n + 2 H) \J\, .
\end{split}
\end{equation}
\end{lemma}

\begin{remark}
In dimensions~$d\geq4$, the Fialkow tensor
is defined by~\cite{Stafford}
\begin{equation}\label{Fial}
{\mathcal F}_{ab}:=\Rho_{ab}^\top-\bar \Rho_{ab}+H\IIo_{ab}+\frac 12\,  \bar g_{ab}H^2\, ,
\end{equation}
and is in fact a weight $w=-2$ tensor density.
Using this and that, via Equation~\nn{one}, three normal derivatives of $I^2_\sigma$ vanishes along $\Sigma$, we may
write the above result as
\begin{equation*}
\begin{split}
 \nabla_n^2 \rho&\stackrel\Sigma=
-\frac{1}{(d-2)(d-3)}\Big(\nablab.\nablab.\IIo +(d-2)(d-4)\IIo^{ab}\bar\Rho_{ab}\Big)\\[2mm]
&-\frac{d-2}{d-3}\, \IIo^{ab}{\mathcal F}_{ab}-\nablab^a\big((\hat n^b \Rho_{ba}){\!}^\top\big)-H\big[(d-2)\Rho(n,n)+\IIo^{ab}\IIo_{ab}\big]+(\nabla_n +  H) \J\, .
\end{split}
\end{equation*}
\end{remark}

In dimensions~$d\geq4$ Lemma~\ref{W3} determines the second normal derivative of~$\rho$ and
when~$d=3$ it gives the obstruction density. The result is the generalization of the 
 Willmore invariant~\nn{Wore} to curved ambient spaces:
\begin{corollary} In dimension~$d=3$, the  obstruction density is given by
\begin{equation}\label{ASCobst2}
{\mathcal B}=-\frac{1}{3}\Big(\nablab_a\nablab_b +H\IIo_{ab}+\Rho_{ab}^\top\Big)\IIo^{ab}\, .
\end{equation}
\end{corollary}

The proof of Proposition~\ref{W3} is involved but conceptually not different to that of Lemma~\ref{lineWillmore}; it is given in Appendix~\ref{nablan3rho}. This confirms 
the result of~\cite{ACF}. The invariant density~$\cB$ was also found using tractor methods in~\cite{YuriThesis}.

\subsection{Constructing hypersurface conformal invariants and holography} 
\label{invts}

The conformal defining density on an ambient manifold  enables a ``holographic'' study of extrinsic as well as intrinsic hypersurface conformal geometry:
The key ingredient is Theorem~\ref{BIGTHE}, which can be used to  proliferate natural invariants  of
the conformal hypersurface structure~$(M,\cc,\Sigma)$ (see Definition~\ref{chi-def}).  Indeed, since the conformal unit defining density~$\si$ is determined by
the data $(M,\cc,\Sigma)$, uniquely modulo~$\O (\si^{d+1})$, up to the
order that~${\si}$ is uniquely determined, the coupled conformal
invariants of the conformal structure and the scale~$ \si$ are
automatically natural invariants of~$(M,\cc,\Sigma)$. Such invariants
are easily constructed using the ambient conformal
tractor calculus applied to~$(M,\cc)$ and~$\si$. Formul\ae\ for conformal  hypersurface invariants obtained by  restricting coupled invariants of the 
ambient structure $(M,\cc,\sigma)$  to~$\Sigma$ 
are termed {\it holographic formul\ae}.

The simplest example of a tractor-valued holographic formula is the restriction of the scale tractor~$I_{\si}$ for a conformal unit defining density~$\si$, which is easily computed using see Lemma~\ref{HII}:
\begin{equation}\label{normaltractor}
I_{\si}^A\big|_\Sigma=N^A\stackrel g=\begin{pmatrix}0\\[1mm] \, \hat n_a\\[1mm]-H\, \end{pmatrix}\, .
\end{equation}
The tractor on the right hand side above is the {\it normal tractor} of~\cite{BEG}.
Therefore, the above is a holographic formula
for the normal tractor.
Another example is the weight $w=-1$, trace-free, symmetric tractor
\begin{equation}
\label{PAB}
P^{AB}:= \hD^A \hD^B \si=\hD^A I_{\si}^B\, ,
\end{equation}
which by construction, for $d\geq 4$, yields a (tractor-valued) hypersurface conformal invariant  upon restriction to $\Sigma$.
In the above, we have used the operator $\hD^A$, which is defined as a map on section spaces $\Gamma(\cT^\Phi M[w])\to
\Gamma(\cT M\otimes \cT^\Phi M[w-1])$ for $w\neq 1-d/2$,  
where $\cT^\Phi M[w]$ denotes a
tractor tensor bundle of arbitrary rank.
In a scale $g$, $$[\hD^A]_g=\big(w,\nabla^\cT_a,-(d+2w-2)^{-1}(g^{ab}\nabla_a^\cT\nabla_b^\cT+ w \J)\big)\, ,$$
and is related to the Thomas D-operator $D^A$ of~\cite{BEG} by $\hD^A=(d+2w-2)^{-1}D^A$. 
Equation~\nn{PAB} can be viewed as   
the conformal analog of Equation~\nn{IIunit}
relating the second fundamental form to a unit defining function. Indeed, in Section~\ref{Lhol} we will use the ambient tractor $P^{AB}$  to build a holographic formula for
the tractor second fundamental form.

We may construct  yet further invariants this way, for example in dimensions~$d\geq 4$, consider the scalar invariant
$$
\big[W_{ABCD}P^{AC}P^{BD}\big]\big|_\Sigma \, ,
$$
where $W_{ABCD}$ is the $W$-tractor of~\cite{GoSrni99} (see also~\cite{GoAdv,GoPetCMP}).
It is an elementary tractor calculus exercise to see that this is simply a multiple of 
$$
W_{abcd}\IIo^{ac}\IIo^{bd} \big|_\Sigma\, ,
$$ where~$W_{abcd}$ is the ambient Weyl curvature. 

It is very
easy to make higher order examples. The key point is that the jets  of objects such as $P^{AB}$ and $W_{ABCD}$
are now canonically defined  (up to the uniqueness bound in the case of $P^{AB}$). 
In particular, the operator obtained through contraction of the scale tractor $I^A$ and the Thomas-$D$ operator $I^A D_A$ gives at the same time (i) a conformal analog of the ambient 
Laplace operator and (ii) along $\Sigma$ a conformally invariant Robin type-operator that can be used to differentiate in the normal direction to the hypersurface. Thus $I\cdot D:=I^A D_A$ is termed the {\it Laplace--Robin operator}, its importance for conformally compact boundary problems is discussed in detail~\cite{GW}. This enables us to perform conformal analogues of comptations such as those leading to Equations~\nn{II2unit} and \nn{trII2unit}.

\subsection{Linking tensor invariants and tractors}

There exists a general ``splitting technology'' (see for example~\cite{Esrni}) relating invariant tensor 
densities and tractors. A particular instance of this is the following construction. First recall that 
there is a canonical bundle inclusion $\ce M[-1] \to \ct M$ given by the canonical section $X^A\in \ct M[1]$. In a scale $g$, $[X^A]_g=(0,0,1)$, and $X^A$ is termed the {\it canonical tractor}. It also induces a surjective bundle map $\cT M[w]\to \ce M[w+1]$ acting on sections by contraction. 
We may extend this to a linear map
~$X_\llcorner:\Gamma(\odot^2 \ct M[w])\to \Gamma(\ct M[w+1])$ acting by contraction with the canonical tractor~$X$.
We now define the canonical map
\begin{equation}\label{q*}
q^*:\ker(X_\llcorner)\to \Gamma(\odot^2T^*M[w+2])\, ,
\end{equation}
which, for some~$g\in \cc$ acts as
$$
q^*:\begin{pmatrix}0&0&0\\0&t^{ab}&t^{a-}\\0&t^{b-}&t^{--}\end{pmatrix}\longmapsto
t_{ab}\, .
$$
The map~$q^*$ can be used to extract  conformal invariants from ambient tractors. 
When interested in hypersurface conformal invariants,
we replace~$\ker(X_\llcorner)$ by~$\ker_\Sigma(X_\llcorner)$ whose elements are tractor sections~$T$ such that~$X_{\llcorner} T=\sigma S$
for some smooth~$S$. This gives a map~$
q_\Sigma^*:\ker_\Sigma(X_\llcorner)\to \Gamma_\Sigma(\odot^2T^*M[w+2])
$,
where~$\Gamma_\Sigma$ denotes equivalence classes of sections~$T\sim \tilde T + \sigma S$ with $S$ smooth. We may identify these with their values along~$\Sigma$.

An application of this construction is  the following result which shows that the tensor~$P^{AB}$ of Equation~\nn{PAB} is the tractor analog of the trace-free second fundamental form while its normal derivative encodes the  invariant Fialkow tensor of Equation~\nn{Fial}.
In the following proposition, we introduce the {\it rigidity density}~$K:=\IIo_{ab}\IIo^{ab}\in\Gamma(\ce \Sigma[-2])$.
\begin{proposition}\label{firsthol}
Let $\sigma$ be a conformal unit defining density, then if~$d\geq 3$,
\begin{equation}\label{theequation}
q^*_\Sigma(P^{AB})=\IIo_{ab}\, \end{equation}
and  $$K_{\rm ext}=P^{AB}P_{AB}\stackrel\Sigma= K\, .$$
For~$ d\geq 4$
\begin{equation}\label{holFialkow}
q_\Sigma^*\, \Big(I\cdot \hD\, P^{AB}+h^{AB}\frac{K_{\rm ext}}{d-2}\Big)^{\!\top}=
-(d-3){\mathcal F_{ab}}+\frac{3\, \bar g_{ab}\, K}{2(d-2)}
\, .
\end{equation}
\end{proposition}
\begin{proof}
First   note that
\begin{equation}\label{XPzero}
X_A P^{AB}=0\, ,
\end{equation}
because~$I=I_\sigma$ has weight zero and~$X\cdot\hD\,  T=wT$ for any weight~$w\neq 1-\frac d2$ tractor. So, in particular $P_{AB}\in \ker_\Sigma(X_\llcorner)$.
Using 
that in a choice of scale $g$, the 
 Thomas D-operator acting on weight $w$ tractors 
is given by 
\begin{equation}
\label{Dform}
[D^A]_g=\big((d+2w-2)w, (d+2w-2)\nabla_a^\ct,-(\Delta+w\J)\big)\, ,
\end{equation}
as well as Formula~\nn{Iform} for the scale tractor,  we see that Equation~\eqref{theequation}  follows  from Corollary~\ref{gradI}. The result for the rigidity density is an immediate consequence.

Next we must  verify that~$I\cdot \hD P^{AB}+h^{AB}\frac{K_{\rm ext}}{d-2}\in\ker_\Sigma(X_{\llcorner})$. 
Acting with~$I\cdot\hD$ on Equation~\nn{XPzero} 
we have
\begin{equation*}\begin{split}
0&\stackrel\Sigma=(I\cdot\hD X_A) P^{AB} + X_A\,  I\cdot \hD P^{AB} 
\stackrel\Sigma=X_A \Big(I\cdot \hD  P^{AB}+h^{AB}\frac{K_{\rm ext}}{d-2}\Big) \, .
\end{split}\end{equation*}
The last equality used~$I\cdot\hD X_A= I{}_A$ and that 
\begin{equation}\label{I.P}
\begin{split}
I^A P_{AB}&=I^A \hD_B I{}_A=
\frac12 \hD_B  I^{\, 2}+\frac{1}{d-2}X_B P^{AC}P_{AC}\stackrel\Sigma=X_B\frac{K_{\rm ext}}{d-2} \, ,
\end{split}
\end{equation}
The second step can be easily explicitly verified or follows from Equation~\nn{failure} below, while
the final step requires $d\geq 3$.

The remainder of the  proof is based on the technology introduced in Section~\ref{comp-ext}. In particular, 
computing along~$\Sigma$:
\begin{eqnarray*}
q_\Sigma^*\Big(I\cdot \hD P^{AB}\!\!&+&h^{AB}\frac{K_{\rm ext}}{d-2}\Big)^{\!\top}\ \stackrel\Sigma=\ 
 \, q_\Sigma^*\Big([\nabla_n+H]P^{AB}\ +\ h^{AB}\frac{K_{\rm ext}}{d-2}\Big)^{\!\top}\\
&\stackrel\Sigma=&\Big[(\nabla_n+H)(\nabla_a n_b + \sigma\Rho_{ab}+\rho g_{ab})+g_{ab}\frac{K}{d-2}\Big]^\top\\
&\stackrel\Sigma=&\Big[\frac12\nabla_a\nabla_b(1-2\rho \sigma)+[\nabla_n,\nabla_a] n_b 
+\Rho_{ab}+\gamma_{ab}\Big(\nabla_n\rho+\frac{K}{d-2}\Big)+H\IIo_{ab}\Big]^\top\\
&\stackrel\Sigma=&
2H\IIo_{ab} +R_{cabd}n^cn^d-\II^2_{ab}
+\Rho^\top_{ab}+\bar g_{ab}\Big(\nabla_n\rho+H^2+\frac{K}{d-2}\Big)\\
&\stackrel\Sigma=&
W_{cabd}\, \hat n^a\hat n^b
-\IIo^2_{ab}+\frac{2\bar g_{ab}K}{d-2}\, .
\end{eqnarray*}
The first two equalities above  use again the explicit formula for Thomas D-operator~\nn{Dform}   and for the scale tractor~\nn{Iform}.
The next line relies on the fact that~$ \sigma$ is a conformal unit defining density and the line thereafter follows directly the method of Section~\ref{comp-ext}
for computing jets of~$ \si$. The last equality required Lemma~\ref{lineWillmore} for the normal derivative of~$\rho$. Finally,
tracing the Gau\ss\ Equation~\nn{Gauss} leads to the following identity 
\begin{equation*}\label{Gausstrace}
\IIo^2_{ab}\!-\frac12\, \bar g_{ab}\scalebox{1.2}{$\frac{\IIo_{cd_{\phantom{a}\!\!}}\!\IIo^{cd}}{d-2}$}-W_{cabd}\, \hat n^c\hat n^d =(d-3)\Big(\Rho_{ab}^\top-\bar \Rho_{ab}+H\IIo_{ab}+\frac 12\,  \bar g_{ab}H^2\Big)\, .
\end{equation*}
The  result follows upon combining the above two displays and the  definition of the Fialkow tensor in Equation~\nn{Fial}.
\end{proof}

\begin{corollary}\label{Kdot}
If~$d\geq 3$, 
$$ I\cdot \hD\,  K_{\rm ext}\stackrel\Sigma=-2(d-3)L\, ,$$
where~$L:=\IIo_{ab }{\mathcal F}^{ab}\in\Gamma(\ce \Sigma[-3])$.
\end{corollary}

\begin{proof}
For $d\geq 4$, the result follows directly from the proposition using the properties of~$P^{AB}$.
For $d=3$, it is easily verified by direct calculation.
\end{proof}

\begin{remark}
Since Equation~\nn{holFialkow} exactly matches~\nn{Gausstrace}, Proposition~\ref{firsthol} allows us to interpret the Fialkow tensor 
as the normal derivative of the trace-free second fundamental form canonically defined by the conformal unit defining density. Later, we will see that quantity $L$ plays the {\it r\^ole} of a rigidity density for embedded volumes.
\end{remark}

The methods used to prove Proposition~\ref{firsthol} can be employed to  generate a set of rank two, symmetric, conformally invariant, extrinsic hypersurface invariants from~$({ I}.\hD)^k P^{AB}$, whose first two elements are the trace-free second fundamental form and the Fialkow tensor.

\section{Conformal hypersurface  tractor calculus}
\label{hollow}

In the previous section we established that conformal hypersurfaces can be naturally treated via tractors. Here we review and extend the known tractor hypersurface calculus
using the conformal unit defining density.
Key results are tractor analogues of the Gau\ss\ formula and second fundamental form. We also show how to relate ambient and hypersurface Thomas D-operators.

\subsection{Tractor second fundamental form}
\label{Lhol}

We first need a certain differential splitting operator~$q$
mapping weighted,  trace-free symmetric two-forms into rank two, weight~$w\neq1-d,-d$, symmetric tractors; this can be viewed as a natural dual of the map $q^*$ defined in Equation~\nn{q*}. For dimensions $d\geq3$ this is given by (see for example~\cite{Esrni}):
$$\Gamma\big(\odot_\circ^2 T^*M[w+2]\big)
\ni \, t_{ab}
\stackrel q\longmapsto
\begin{pmatrix}0&0&0\\[1mm]0&t^{ab}&-\frac{ \nabla.t^a}{d+w}\\[1mm]
0&-\frac{ \nabla.t^b}{d+w}&\frac{\nabla.\nabla.t+(d+w)\Rho_{ab}t^{ab}}{(d+w)(d+w-1)}\end{pmatrix}
=:[T^{AB}]\in \Gamma\big(\ct^{(AB)_\circ} M[w]\big) \, .$$
When in addition $w\neq-\frac d2$, the conditions
$
D_A T^{AB}=0=X_A T^{AB}=T_A{}^A
$ characterise the image of this map.

\begin{remark}
When~$t_{ab}\in 
\Gamma\big(\odot_\circ^2 T^*M[3-d]\big)$,
 the weight $-d-1$~density $$\big(\nabla_a\nabla_b+\Rho_{ab}\big) t^{ab}\, ,$$
appearing as the residue of the pole at $w=1-d$ in the above display,
is  conformally invariant.
\end{remark}

On the conformal manifold~$(\Sigma,\cc_\Sigma)$,
applying the map~$q$
 to  the trace-free second fundamental form~$\IIo_{ab}\in \Gamma(\odot^2_\circ T^*\Sigma[1])$ gives the tractor second fundamental form~\cite{Grant,Stafford}:
 \begin{definition}
 Let~$d\geq 4$.
The tractor second fundamental form~$L^{AB}\in \Gamma\big(\ct^{(AB)_\circ} \Sigma[-1]\big)$ is defined by
\begin{equation}\label{LAB}
[L^{AB}]:=q(\IIo_{ab})\stackrel{\bar g\in\cc_{\Sigma}}=\begin{pmatrix}0&0&0\\[1mm]0&\IIo_{ab}&-\frac{ \nablab.\IIo_a}{d-2}\\[1mm]
0&-\frac{ \nablab.\IIo_b}{d-2}&\frac{\nablab.\nablab.\IIo+(d-2)\bar P_{ab}\IIo^{ab}}{(d-2)(d-3)}\end{pmatrix}
\, .
\end{equation}
 \end{definition}

\begin{remark}\label{dimreg}
A dimensional continuation argument can be used to obtain the~$d=3$ obstruction density from the 
tractor second fundamental form: In dimensions~$d\geq 4$, the  Fialkow--Gau\ss\ equation~\nn{Gausstrace} implies
$$\bar P_{ab} \IIo^{ab}=P_{ab}^\top\IIo^{ab}+H\IIo_{ab}\IIo^{ab} + \frac{\IIo^{ab}W_{cabd}\, \hat n^c \hat n^d-\IIo_{ab}\IIo^{ac}\IIo_c^b}{d-3}\, ,~$$
so that the part of~$L^{AB}$ singular when~$d=3$ can be rewritten as
$$
\frac{\nablab.\nablab.\IIo+(d-2)(\Rho_{ab}^\top + H \IIo_{ab})\IIo^{ab}}{(d-2)(d-3)}+\frac{\IIo^{ab}W_{cabd}\, \hat n^c \hat n^d-\IIo_{ab}\IIo^{bc}\IIo_c^a}{(d-3)^2}\, .
$$
The numerator of the second term in this expression vanishes identically in~$d=3$ while the first numerator
evaluated at~$d=3$ is 
$$
\big(\nablab_a \nablab_b +\Rho_{ab}^\top + H \IIo_{ab}\big)\IIo^{ab}\, ;
$$
this is precisely the obstruction density~\nn{ASCobst2}.
\end{remark}

Corollary~\ref{gradI} and Proposition~\ref{firsthol} suggest that a holographic formula for the tractor second fundamental form
can be built from~$P^{AB}=\hD^A  I^B$. For that, we need the following result:

\begin{lemma}
Let~$\sigma$ be a conformal unit defining density and~$d\geq4$, then for~$g\in\cc$,
$$
 [\Delta
 I^A_{}]\stackrel\Sigma=
\begin{pmatrix}
0\\[2mm]
\nablab.\IIo_a-\hat n_aK\\[1mm]
-\frac{\nablab.\nablab.\IIo + (d-2)\IIo^{ab}\bar\Rho_{ab}}{d-3}
+2HK-\frac{(3d-8)L}{d-3}
\end{pmatrix}\, .
$$

\end{lemma} 

\begin{proof}
Firstly, recall that in a choice of scale $g$, the tractor connection acts on a standard tractor $V^A$ according to (see for example~\cite{BEG})
\begin{equation}\label{trconn}
\nd^{\ct}_a
\left( \begin{array}{c}
v^+\\[1mm]v_b\, \\[1mm] v^-
\end{array} \right) \stackrel g=
\left( \begin{array}{c}
    \nabla_a v^+-v_a \\[1mm]
    \nabla_a v_b+ g_{ab} v^- +\Rho_{ab}v^+ \\[1mm]
    \nabla_a v^- - \Rho_{ac}v^c  \end{array} \right) .
\end{equation}
Applying the above equation to the scale tractor twice and then contracting with the inverse metric yields
$$
[g^{ab}\nabla_a\nabla_b I^A_{\sigma}]\stackrel\Sigma=
\begin{pmatrix}
0\\[1mm]
\Delta n_a +2\nabla_a \rho \\[1mm]
(\Delta-\J)\rho -2\Rho^{ab}\nabla_a n_b-\nabla_n \J
\end{pmatrix}\, .
$$
Along~$\Sigma$ we have
\begin{eqnarray*}
\Delta n_a+2\nabla_a\rho&=&
\nabla^b \nabla_a n_b+2\nabla_a\rho\\
&=&\nabla_a \nabla.n + \Ric_{ab}n^b+2\nabla_a\rho\\
&=&-(d-2)\big(\nabla_a \rho - \Rho_{ab}n^b\big)\\
&=&-(d-2)\big(\nabla_a^\top \rho + n_a \nabla_n \rho -\Rho_{ab}n^b\big)\\
&=&-n_a\IIo_{bc}\IIo^{bc}+(d-2)\big(\nablab_a H + (\Rho_{ab}\hat n^b)^\top\big)\, .
\end{eqnarray*}
The last line was obtained using Equations~\nn{rhoH} and~\nn{nablanrho}.
The traced Codazzi-Mainardi equation~\nn{Mainarditrace} 
establishes the middle slot of the right hand side of the displayed result. Note that this result could also be obtained from Corollary~\ref{gradI} and symmetry of~$P^{AB}$.

Also, computing along~$\Sigma$ (using Lemma~\ref{Laplaces} to handle the ambient Laplace operator and Equation~\nn{IIc} for the gradient of the normal vector)
\begin{equation*}
\begin{split}
(\Delta-\J)\rho& -2\Rho^{ab}\nabla_a n_b-\nabla_n \J\\
&=-\bar \Delta H +\nabla_n^2\rho +(d-2) H \nabla_n\rho-2\Rho^{ab}\IIo_{ab}-(\nabla_n+H)\J\, .
\end{split}
\end{equation*}
Normal derivatives of~$\rho$ are given  by Lemma~\ref{lineWillmore} and Proposition~\ref{W3}. Furthermore,
a simple consequence of the Codazzi--Mainardi Equation~\nn{Mainardi} is the following identity
\begin{equation}\label{boxH}
\begin{split}
\bar\Delta H&=
\frac1{d-1}\nablab^a\big(\nablab.\II_a-(\Ric_{ab}\hat n^b)^\top\big) =\frac{1}{d-2}\nablab^a\big(\nablab.\IIo_a-(d-2)(\Rho_{ab}\hat n^b)^\top\big)\, ,
\end{split}
\end{equation}
which allows
 the Laplacian of the mean curvature to be traded for 
divergences of the trace-free second fundamental form. In addition, normal derivatives of the normal components of the Einstein tensor are given by Lemma~\ref{Einstein} and the ambient Schouten tensor can be eliminated using the Fialkow--Gau\ss\ equation~\nn{Gausstrace}. Orchestrating those maneuvers gives the bottom slot of
the displayed result and completes the proof.
\end{proof}

The above lemma combined with Corollary~\ref{gradI} determine $\hD^A  I^B$ along~$\Sigma$. This, together with
 Corollary~\ref{Kdot}, gives the following  holographic formula for the tractor second fundamental form (up to a slight modification): 
\begin{proposition}\label{HOLII}
Let $\sigma$ be a conformal unit defining density and~$d\geq4$. Then
\begin{equation}\label{holII}
\left.\left[\hD^A  I^B-\frac{2}{d-2}  I^{(A}X^{B)} K_{\rm ext}+\frac{X^AX^B\,  I\cdot \hD K_{\rm ext}}{(d-2)(d-3)}\right]\right|_\Sigma=
L^{AB}\, + \, \frac{X^A X^B\, L}{d-3}\, .
\end{equation}
\end{proposition}

\begin{remark}
The first term on the left hand side of~\nn{holII} is~$P^{AB}$ as promised in Section~\ref{invts}. 
It follows from Equation~\nn{I.P} that, along~$\Sigma$,  the first two terms are the orthogonal projection of $P^{AB}$
to hypersurface tractors (meaning sections of the tractor subbundle consisting of tractors orthogonal to the normal tractor).
The failure of this 
 to be a holographic formula for the tractor second fundamental form is measured by $ I\cdot  \hD K_{\rm ext}\big|_\Sigma=-2(d-3) L$, which equals the contraction of the Fialkow tensor and the trace-free second fundamental form, see Corollary~\ref{Kdot}.
\end{remark}

\subsection{Thomas D-operator}

Here, given a defining density for a hypersurface~$\Sigma$, we construct a general family of  tangential operators (this notion was introduced in~\cite{GW} to describe ambient operators that descend to hypersurface operators upon restriction; see Definition~\ref{tangential} below) that relate the ambient and intrinsic Thomas D-operators along~$\Sigma$. 
The following Proposition was proved in~\cite{GLW} for the special case of the AE setting:

\begin{proposition}\label{DT}
Let~$\sigma$ be a defining density for a hypersurface~$\Sigma$
and denote
$
\hat I^A:=I_\sigma^A/\sqrt{I^2_\sigma}
$.
Then, if~$w+\frac d2\neq 1,\frac32,2$, the operator
\begin{equation}\label{DTanysigma}
\hD^T_A:=\hD_A-\hat I_A \hat I\cdot\hD +\frac{I^2}{h(h-1)(h-2)}\, X_A\big(\frac{1}{I^2} I\cdot D\big)^2\, ,\quad h+2:=d+2w\, ,
\end{equation}
mapping~$\Gamma(\ct^\Phi M[w])\to \Gamma(\ct_AM\otimes \ct^{\Phi}  M[w-1])$, is tangential.
\end{proposition}

\begin{proof}
The proof of this result only requires that we establish the operator relation
$$
\hD^{\rm T}_A\,\circ  \sigma \stackrel\Sigma = 0\, .
$$
This follows from two facts:  (i) The $\frak{sl}(2)$ algebra (see~\cite{GW})
\begin{equation}\label{sl2}
[d+2\w,\sigma]=2\sigma\, ,\quad
\Big[\frac1{I^2} I\cdot D,\sigma\Big]=d+2\w\, ,\quad
\Big[d+2\w, \frac1{I^2} I\cdot D\Big]=-\frac2{I^2} I\cdot D\, ,
\end{equation}
 spanned by~$\sigma$ (viewed as a multiplicative operator on sections), $d+2\w$ where $\w$ is the linear operator that returns the weight of a tractor, and~$\frac1{I^2} I\cdot D$. (ii) The commutator of~$\hD^A$ and~$\sigma$ (again viewed as a multiplicative operator) 
$$
[\hD^A,\sigma]=I^A-\frac{2}{h(h-2)}X^A I\cdot D\, ,\quad h:=d+2w\, ,
$$
valid acting on tractors of weight~$w\neq -\frac d2,1-\frac d2$ which is easily verified by direct computation in a choice of scale.
\end{proof}

\begin{remark}\label{DTYam}
In fact we will also need a replacement of the tangential Thomas D-operator 
at the missing weight $w=1-\tfrac d2$.
Given a weight $w'$ tractor $V^A
\in \Gamma(\ct^AM\otimes \ct^{\Phi'} M[w'])$ subject to
$
X_A V^A=0\, ,
$
and
$$
N_A V^A\stackrel\Sigma=0\, ,
$$
we can construct a tangential analog of the operator $V^A \hD_A^T$   at the {\it Yamabe weight} $w=1-\frac d2$ as follows:  First, calling $[V^A]_g=(0,v_a,v)$, it is easy to check that the operator, given by
$$
V^A \hD_A^T \stackrel g{:=} v^a\nabla_a + \big[1-\frac d2\big] \, v\, ,
$$ 
for some $g\in \cc$ 
 defines a mapping $\Gamma(\ct^\Phi M[1-\frac d2])\to \Gamma(\ct^{\Phi'}M[w']\otimes \ct^{\Phi}  M[-\frac d2])$.
 However, for {\it any} defining density $\sigma$, we have $I_\sigma \cdot V=\sigma u$ for some smooth, weight $w'-1$ density~$u$. Thus, for some $g\in \cc$ we have $n^av_a+ \sigma v=\sigma u$ and hence $V\cdot \hD^T\stackrel g=v^a\nabla^\top_a+v\, \big[1-\frac d2\big]$
$ +\ {\mathcal O}(\sigma)\nabla_{\hat n}$, which is clearly tangential. 
 \end{remark}

\medskip

Proposition~\ref{DT} suggests that when expressed in terms of a scale,  the tangential Thomas~D-operator~
$$D^T_A:=\left\{
\begin{array}{ll}
(d+2w-2)\hD_A^{\rm T}\, , & w\neq 1-\frac d2,\frac32-\frac d2, 2-\frac d2\, ,\\[3mm]
D_A-\hat I_A \hat I\cdot D +  X_A I\cdot D \circ \frac{1}{2I^2} \circ I\cdot D\, ,&
w=1-\frac d2\, ,
\end{array}\right.
$$ depends on the tractor-coupled connection
only through the tangential combination $\nabla^\top_a:=\nabla_{a}-\hat n_a \nabla_{\hat n}$.
For the case where the defining density is  conformal unit,
it follows immediately that the operator $\hD_A^T$ is independent of any choices.
For that case we
call the operator $D^T_A$ the {\it tangential Thomas D-operator}. We will  verify that this operator indeed factors through $\nabla_a^\top$ in the sense mentioned, see  Equation~\nn{DTA} of the following Lemma.  
That Lemma also  collects 
a number of  critical  results and details important for later 
developments. Let us point out some  interesting features: 
 In Equation~\nn{ID2}, the general formula for $(I\cdot D)^2$ along~$\Sigma$ is given; for {\em boundary Yamabe weight}~$w=\frac32-\frac d2$, all normal derivatives drop out, implying that this operator then becomes tangential. This is the first example of the extrinsic conformal Laplacians discussed in Section~\ref{invops} and Remark~\ref{Yamrem}.  The Lemma's next equation specialises Equation~\nn{DTanysigma} to conformal unit defining densities. The formula for this
 in a choice of scale, given in Equation~\nn{DTA}, should be compared with the general result for the Thomas D-operator in Equation~\nn{Dform}, keeping in mind that the  orthogonal subbundle~$N^\perp$ of $\cT M|_\Sigma$ and the intrinsic hypersurface  tractor bundle~$\cT \Sigma$ are isomorphic
 (see~\cite{BrGoCNV,Grant} as well as~\cite[Section 3.2]{GW15} for details). This shows that, along~$\Sigma$, the tangential Thomas D-operator yields an extrinsic hypersurface Thomas D-operator with ambient tractor-coupled connection save for a modification by   the operator~$\frac{wKX^A}{2(d-2)(d+2w-3)}$. 

\begin{lemma}\label{TANGD}
Let~$\sigma$ be a conformal unit defining density and $d\geq 3$. 
Then, acting on weight~$w$ tractors, the following operator identity holds along~$\Sigma$, in a choice of scale~$g$,
\begin{equation}\label{ID2}
\begin{split}
(I\cdot D)^2 \stackrel\Sigma=&-(d+2w-4)\!\left\{\Delta^{\!\!\top}
+w\, \Big[\bar \J-\frac12\frac{\IIo_{ab}\IIo^{ab}}{d-2}\Big]\right.\\&
\left.-\, (d+2w-3)\ \Big[
\nabla_n^2-w\Big(2H\nabla_n -\Rho(n,n)-\frac{\IIo_{ab}\IIo^{ab}}{d-2}-\frac{(2w-1)H^2}2 
\Big)\Big]
\right\}\, ,
\end{split}
\end{equation}
where $\Delta^{\!\top}:=\bar g^{ab}\nabla_a^\top\nabla_b^\top$.
Moreover, specializing to a conformal unit defining density,
the operator~$\hD_A^{\rm T}$, as defined in Proposition~\ref{DT} , is given by
\begin{equation}\label{DTAdef}
\hD^{\rm T}_A=
\hD_A- I{}_A  I\cdot\hD +\frac{1}{h(h-1)(h-2)}\, X_A\big( I\cdot D\big)^2\, ,\quad h+2:=d+2w\, .
\end{equation}
It is determined up to terms of order $\O(\sigma^{d-1})$ times a smooth differential operator, and is subject to the same weight restrictions as in Proposition~\ref{DT}. In a choice of scale~$g$,
\begin{equation}\label{DTA}
\big[\hD^{{}^{\rm T}}{\!}^A\big]_g
\stackrel\Sigma=
\begin{pmatrix}
1&0&0\ \\[1mm]
\, n_a H&\delta_a^b&0\ \\[1mm]
-\frac{H^2}{2}&\!-n^b H\, &1\ 
\end{pmatrix}
\left[
\begin{pmatrix}
w\\[1mm]
\nabla^\top_b\\[1mm]
-\frac{\Delta^{\!\top} +w\bar J}{d+2w-3}
\end{pmatrix}+\begin{pmatrix}0\\[1mm]0\\[1mm]\frac{w\IIo_{ab}\IIo^{ab}}{2(d-2)(d+2w-3)}
\end{pmatrix}\right]\, .
\end{equation}
\end{lemma} 
\begin{proof} 
For the first statement, 
we first use 
that
$$
I\cdot D\stackrel g=(d+2w-2)(\nabla_n+w\rho)
-\sigma(\Delta+w\J)
$$
and
\begin{equation}\label{RobinI}
 I\cdot\hD\stackrel\Sigma=\nabla_n-wH
\end{equation}
to compute the operator statement (acting on weight~$w$ objects)
 along~$\Sigma$ directly
\begin{equation*}
\begin{split}
& I\cdot D^2\stackrel\Sigma=
(d+2w-4)\big[\nabla_n-(w-1)H\big]\big[(d+2w-2)(\nabla_n+w\rho)-\si (\Delta+w\J)\big]\\
&=\!-(d+2w-4)\Big[\Delta\!+w\J\!-(d+2w-2)\big(\nabla_n^2\!-\!(2w-1)H\nabla_n\!+w(\nabla_n\rho)\!+\!w(w-1)H^2\big)\Big]\, .
\end{split}
\end{equation*}
On the second line we used  the operator product identity~$\nabla_n \circ  \si =1$ valid along~$\Sigma$ and~$\rho|_\Sigma=-H$ as per Lemma~\ref{HII}.
To obtain the quoted result we used Equation~\ref{JJbar}, Lemma~\ref{lineWillmore} as well as the operator identity for the tractor-coupled Laplacian \begin{equation}\label{opid}\Delta\stackrel\Sigma=\Delta^{\!\top}+\nabla_n^2+(d-2)H\nabla_n\, ,\end{equation}
which can easily be established along the same lines used to prove Lemma~\ref{Laplaces}.

The second statement  follows from the defining property of a conformal unit defining density in   Equation~\nn{one}.  
For the third 
we first use Equation~\nn{Dform} as well as
Equation~\nn{RobinI}
to find the operator statement for the first two terms of Equation~\nn{DTAdef},
$$
\big[\hD^A -  I^A I\cdot \hD\big]_g\stackrel\Sigma=
\begin{pmatrix}
w\\[1mm]
\nabla_a-n_a (\nabla_n-wH)\\[1.5mm]
-\frac{1}{d+2w-2}\big(\Delta+\J w\big)+H(\nabla_n-wH)
\end{pmatrix}\, .
$$
Remembering that $\nabla_a^\top\stackrel\Sigma= \nabla_a- n_a \nabla_n$, it is easy to verify that the top two slots on the right hand side of the above display agree with those quoted in Equation~\nn{DTA}. Thus it only remains to verify the bottom slot of Equation~\nn{DTA}. Using the computation of $ I.D^2$ along~$\Sigma$ shown above, as well as the bottom slot in the above display, 
one can 
employ Equation~\nn{JJbar}
to trade the ambient~$\J$ for its intrinsic counterpart~$\bar\J$, Equation~\nn{opid} to exchange $\Delta$ for $\Delta^{\!\top}$ and Equation~\nn{nablanrho}
to handle $\nabla_n\rho$. This yields the result quoted in the bottom slot of Equation~\nn{DTAdef}. 
\end{proof}

The Thomas D-operator identity
\begin{equation}\label{DXT}
D_A\circ X^A = (d+\w)(d+2\w+2)
\end{equation}
is useful in many contexts; the tangential 
 Thomas D-operator obeys an analog of this:
  \begin{corollary}
Let $T\in \Gamma(\ct^\Phi M[w])$ where $w+\frac d2\neq1,\frac32,2$. Then
\begin{equation}\label{bDXT}
\hD^T_A\big(X^A T)\stackrel\Sigma= \frac{(d+2w+1)(d+w-1)\,  T}{d+2w-1}\, .
\end{equation}
\end{corollary}

\begin{proof}
Noting that in a choice of scale, 
$$
\big[\nabla^\top_b \big(X^A T)\big]_g\stackrel\Sigma=
\begin{pmatrix}
0\\
\bar g_{ab} T\\[.5mm]0
\end{pmatrix}\, \mbox{ and }
\big[\Delta^\top \big(X^A T\big)\big]_g\stackrel\Sigma=
\begin{pmatrix}
(d-1)T\\\star\\\star\end{pmatrix}\, ,
$$
the result follows directly by application of Equation~\ref{DTA}.
\end{proof}

\begin{remark}\label{Yamrem}
As mentioned above (see also~\cite{Us,GLW}) at weight~$w=\frac32-\frac {d}2$, the terms in  Equation~\nn{ID2} above  involving~$\nabla_n$ are absent, and the operator~
$$ I\cdot D^2\stackrel \Sigma=
\Delta^{\!\top}+\Big(\frac32-\frac d2\Big)\Big[\bar \J-\frac12\frac{\IIo_{ab}\IIo^{ab}}{d-2}\Big]=:\square^\top_{\rm Y}\, ,$$
 is
 tangential. 
Specializing to densities,~$\Delta^{\!\top}$ becomes the intrinsic Laplace operator~$\bar \Delta$ along~$\Sigma$ and  $\square^\top_{\rm Y}$ is the intrinsic Yamabe Laplacian modified by 
the rigidity density.
\end{remark}

Our first application of the canonical tangential Thomas D-operator is to compute its action on the scale tractor. This gives another holographic formula for the tractor second fundamental form (again up to known terms) that can  be regarded as a conformal analog of the Riemannian result  for the second fundamental form in terms of the ambient Levi-Civita connection acting on a unit normal vector
 in Equation~\nn{two}.

\begin{proposition}\label{DTII}
Let $d\geq 4$. Then
$$
\hD^{\rm T}_A  
N_B^{\phantom{\rm T}}\stackrel\Sigma=L_{AB}+\frac{X_A(N_B\, K + X_B L)}{d-3}\, ,$$
where $N_B$ is any smooth extension of the normal tractor off~$\Sigma$.
\end{proposition}

\begin{proof}
Let $\si$ be a conformal unit defining density and $ I_A:=\hD_A\sigma$. Then, since ${\hD^T}_A$ is tangential, we may replace the left hand side of the above display by ${\hD^T}_A  I_B^{\phantom{\rm T}}$, which we shall now compute.
From Corollary~\ref{gradI} we have
$$
[\nabla^\top_a  I^B]\stackrel\Sigma=
\begin{pmatrix}
0\, \\ \IIo_{ab} \\[1mm] -\frac{\nablab.\IIo_a}{d-2}\ 
\end{pmatrix}\, .
$$
Using that~$\nabla^\top_a \IIo^{a}_{b}=\nablab.\IIo_b-n_b \II_{ac}\IIo^{ac}=\nablab.\IIo_b-n_b K$, we compute
$$
[\Delta^{\!\top}  I^B]=\begin{pmatrix}
0\\[1mm] \frac{d-3}{d-2}\nablab.\IIo_b-n_b K
\\[2mm]
-\frac{\nablab.\nablab.\IIo+(d-2)\Rho_{ab}\IIo^{ab}}{d-2}
\end{pmatrix}\, .
$$ 
We now have the main ingredients required to employ Equation~\nn{DTA} of Lemma~\ref{TANGD}, and find
$$
[\hD^{{}^{\rm T}}{\!}^A  I^B]\stackrel\Sigma=
\begin{pmatrix}0&0&0\\
0&\IIo_{ab}&
-\frac{\nablab.\IIo_a}{d-2}\\[1mm]
0&-\frac{\nablab.\IIo_b}{d-2}+\frac{n_b K}{d-3}&
\frac{\nablab.\nablab.\IIo+(d-2)\Rho_{ab}\IIo^{ab}}{(d-2)(d-3)}
\end{pmatrix}\, .
$$
The final result is obtained upon using the Fialkow--Gau\ss\ Equation~\nn{Gausstrace}
to give~$$\Rho_{ab}\IIo^{ab}=\bar\Rho_{ab}\IIo^{ab}+L-HK\, .$$
\end{proof}

\begin{remark}
Since Propositions~\ref{HOLII} and~\ref{DTII} 
both give holographic formul\ae\ for the
tractor second fundamental form, 
we can use the former to given an alternate proof of the latter, without recourse to explicit expressions in a choice of scale: One begins by 
using Equations~\nn{PAB} and~\nn{DTAdef} to give
$
\hD^{\rm T}_A I_B=P_{AB}- I_A  I\cdot P_B +(d-3)^{-1}X_A  I\cdot\hD \,   I\cdot P_B\
$ (for $d\geq 4$). Then employing Equation~\nn{I.P}
in concert with Corollary~\ref{Kdot} and Proposition~\ref{HOLII}, the result of Proposition~\ref{DTII} can easily  be obtained by applying the fundamental calculus of the Thomas D-operator expressed by the modified Leibniz rule~\cite{Joung}
\begin{equation}\label{failure}
\hD^A(T_1T_2)- (\hD^A T_1) \, T_2 - T_1(\hD^A T_2)=-\frac{2}{d+2w_1 + 2w_2 -2}\, X^A\, (\hD_B T_1)(\hD^B T_2)\, ,
\end{equation}
valid for 
$T_{1,2}\in \Gamma(\cT^\Phi M[w_{1,2}\neq-d/2])$ and $w_1+w_2\neq 1-d/2$,
and the resulting operator commutator relation (see~\cite[Section 3.6]{GW15})
\begin{equation} \label{hDsip}
[\hD^A,\sigma^k] = k\,  \sigma^{k-1}I^A 
- \frac{2k\,  X^A\sigma^{k-1}I\cdot D }{(d\!+2k\!+2w\!-2)(d+2w-2)}
-\frac{k(k-1)X^A\sigma^{k-2}I^2  }{d+2k+2w-2} \, \quad k\in {\mathbb Z}_{\geq 0},
\end{equation}
valid for {\it any} scale $\sigma$ and acting on tractors of weight $w\neq 1-d/2,1-k-d/2$.
\end{remark}

To complete the relationship between the tangential Thom\-as-$D$~operator $D^T_A$ and the intrinsic Thomas D-operator~$\bar D_A$ of the hypersurface~$\Sigma$, we need a generalization of the Gau\ss\ formula~\nn{hypgrad} relating the projected  tractor connection~$\nabla^\top$ to its intrinsic hypersurface counterpart~$\nablab$ (this result was also developed in~\cite{Grant, Stafford, YuriThesis,Calderbank}).

\begin{proposition}[Fialkow--Gau\ss\ formula]\label{FGformula}

Let~$V^A\in \Gamma(\ct M)$
be such that along~$\Sigma$ it lies in
  $ \Gamma(N^\perp)$ and denote
 by  $\Sigma^A_B:=\delta^A_B -N^A N_B$  the projector mapping~$\Gamma\big(\ct M\big|_\Sigma\big)\to \Gamma(N^\perp)$. Then, for $d\geq 4$,
$$
\Sigma^A_B \nabla_c^\top V^B =  \nabla_c^\top V^A+N^A L_c^B V_B
= \bar\nabla_c \bar V^A + \F_c{}^A{}_B \bar V^B=:\nablab_c^\F \, \bar V^A\, .$$
Here ${\mathcal F}$ is  a conformally  invariant, one-form valued, boundary tractor endomorphism  given in a boundary splitting by
$$
\big[\F_c{}^A{}_B\big]_{\bar g}=\begin{pmatrix}
0&0&0\\
\F_{ca}&0&0\\
0&-\F_c{}^b
&0
\end{pmatrix}\, .
$$
\end{proposition}

\begin{proof} 
Let us fix an ambient scale $g\in \cc$. This induces a boundary scale $\bar g\in \cc_\Sigma$. Now recall (see~\cite[Section 3.2]{GW15})
that the isomorphism between
 the subbundle~$N^\perp$ orthogonal to the normal tractor (with respect to the tractor metric~$h$)
along~$\Sigma$
and 
the intrinsic hypersurface tractor bundle~$\ct\Sigma$ gives a map between
sections 
 expressed in  scales~$g$ and~$\bar g$, respectively:
\begin{equation}\label{Tisomorphism}
\big[V^A\big]_g:=\begin{pmatrix}v^+\\[1mm]v_a\, \\[2mm]v^-\end{pmatrix}\stackrel{\cong}\longmapsto
\begin{pmatrix}
v^+\\[1mm]v_a-\hat n_aH v^+\\[2mm]v^-+\frac12 H^2 v^+
\end{pmatrix}=\big[U^A{}_B\big]_{\bar g}^g\, \big[ V^B\big]_g=:\big[\bar V^A\big]_{\bar g}\, ,
\end{equation}
where~$V^A\in\Gamma (N^\perp)$ and~$\bar V^A\in  \Gamma(\ct \Sigma)$.
Here the~$SO(d+1,1)$-valued matrix
$$
\big[U^A{}_B\big]_{\bar g}^g:=\begin{pmatrix}1&0&\ 0\ \\[2mm]-\hat n_a H&\delta_a^b&0\\[3mm]-\frac12 H^2&\hat n^b H&1\end{pmatrix}\, ,
$$
and  
the unit conormal $\hat n$ has been used to 
identify sections of~$T^*M\big|_\Sigma$ and~$T^*\Sigma$.
(Note that 
the map in Equation~\nn{Tisomorphism} is the identity
for tractors along~$\Sigma$  in the joint kernel of the contraction maps~$X_\llcorner$ and~$N_{\llcorner}$.)

Thus, in terms of the above ambient and boundary splittings  we need need to show that
$$
\big[ U^B{}_C\big]_{\bar g}^g \, \big[\Sigma^C_D \nabla_a^\top V^D\big]_g =  
\big[ U^B{}_C\big]_{\bar g}^g\, 
\big[\nabla_a^\top V^C+N^C L_a^D V_D\big]_g
= \big[\bar\nabla_a \bar V^B + \F_a{}^B{}_C \bar V^C\big]_{\bar g}\, .$$
 Now, along~$\Sigma$, we have (using the expression for the normal tractor in Equation~\nn{normaltractor})  that 
 $V_A=(v^-,v_a,v^+)\in\Gamma(N^\perp)$ obeys $\hat n.v\stackrel\Sigma=Hv^+$ while the isomorphism
 between $\Gamma(N^\perp)$ and $\Gamma(\ct \Sigma)|_\Sigma$, given 
 in scales $(g,\bar g)$ in Equation~\nn{Tisomorphism}, maps $V_A$ to $\bar V_A=(\bar v^-,\bar v_a,v^+)$ where
 $$\bar v_a\stackrel\Sigma=v^\top_a\, ,\qquad \bar v^-\stackrel\Sigma=v^-+\frac12 H^2 v^+\, .$$
Using  the expression for the tractor connection acting on a standard tractor in Equation~\nn{trconn} applied to our   choice of ambient  scale~$g$ we have
$$
[\nabla_a^\top V^B]\stackrel\Sigma=\begin{pmatrix}
\nabla_a^\top v^+-v_a^\top\\[1mm]
\nabla_a^\top v_b + (\Rho_{ab}-\hat n_a \Rho(\hat n,b))v^+ + (g_{ab}-\hat n_a\hat n_b)v^-\\[1mm]
\nabla_a^\top v^--\Rho(a,v)+\hat n_a \Rho(\hat n,v)\end{pmatrix}\, .
$$
We now simplify, slot by slot, each expression on the right hand side, beginning at the top:
$$
\nabla_a^\top v^+-v_a^\top\stackrel\Sigma=\nablab_a v^+-\bar v_a\, .
$$
For the middle slot we have
\begin{equation*}
\begin{split}
\nabla_a^\top v_b &+ (\Rho_{ab}-\hat n_a \Rho(\hat n,b))v^+ + (g_{ab}-\hat n_a\hat n_b)v^-\\[2mm]
&=
\nablab_a \bar v_b-\hat n_b \II_a^c \bar v_c+\II_{ab} H v^+ + \hat  n_b (\nablab_a H) v^+ + \hat n_b H\nablab_a v^+\\&\quad
+\Rho_{ab}^\top v^+  +\hat n_b \Rho(\hat n,a) v^+ -\hat n_a \hat n_b\Rho(\hat n,\hat n)v^++\bar g_{ab} v^-\\[2mm]
&=
\nablab_a \bar v_b
-\hat n_b\Big(\IIo_{ac}\bar v^c- \frac{(\nablab.\IIo_a) v^+}{d-2}\Big)
+\hat n^b H \big(\nablab_a v^+ -\bar v_a\big)\\
&
+\Big(\Rho_{ab}^\top 
+H\IIo_{ab} +\frac12 \bar g_{ab} H^2\Big) v^+ 
  +\bar g_{ab} \bar v^-\, .
\end{split}
\end{equation*}
Here we have used the traced-Gau\ss--Mainardi Equation~\nn{Mainarditrace} to handle  gradients of mean curvature.
Observe that for $d\geq 4$, the 
last term in brackets, by virtue of the Gau\ss--Fialkow Equation~\nn{Gausstrace}, becomes simply $\bar \Rho_{ab}+\F_{ab}$. So the middle slot is
$$
\nablab_a \bar v_b + (\bar \Rho_{ab}+\F_{ab}) \bar v^+
 +\bar g_{ab} \bar v^-
-\hat n_b\Big(\IIo_{ac}\bar v^c- \frac{(\nablab.\IIo_a) v^+}{d-2}\Big)
+\hat n^b H \big(\nablab_a v^+ -\bar v_a\big)
  \, .
$$  
For the bottom slot we have, using the same method at $d\geq 4$,
\begin{equation*}
\begin{split}
\nabla_a^\top v^-&-\Rho(a,v)+\hat n_a \Rho(\hat n,v)\\[1mm]&=
\nablab_a \bar v^-
 -(
 \bar \Rho_{ab} +\F_{ab})\bar v^b
-\frac12 H^2 (\nablab_a v^+-\bar v_a) 
 +H\Big(\IIo_{ab}\bar v^b
- \frac{1}{d-2}\nablab.\IIo_av^+\Big)
\, .
\end{split}
\end{equation*}
Putting the three slots back together, we find that  
~$[\nabla_a^\top V^B]$  (along~$\Sigma$) is
\begin{eqnarray*}
\begin{pmatrix}
1&0&0\ \\[2.5mm]
\, \hat n_b H&\delta_{b}^c&0\ \\[2.5mm]
-\frac{H^2}{2}&\!-\hat n^c H\, &1\ 
\end{pmatrix}
\left[
\begin{pmatrix}
\nablab_a v^+-\bar v_a\\[3mm]
\nablab_a \bar v_c+\bar\Rho_{ac} v^++ 
\bar g_{ac} \bar v^-\\[3mm]
\nablab_a \bar v^- -\bar\Rho_{ab} \bar v^b  
\end{pmatrix}
-
\begin{pmatrix}
0\\[3mm]
\hat n_c\\[3mm]0
\end{pmatrix}L_a^D \bar V_D+
\begin{pmatrix}
0\\[3mm]
\, \F_{ac} \, v^+\\[3mm]
-\F_{ab} \, \bar v^b\ \end{pmatrix}
\right]\, ,
\\[3mm]
\end{eqnarray*}
where (according to Equation~\nn{LAB})~$L_a^D \bar V_D=\IIo_{ad}\bar v^d- \frac{(\nablab.\IIo_a)  v^+}{d-2}$. This establishes the second equality  displayed at the beginning of the proof. It remains to establish the first equation shown there. For that note  that 
$$
\Sigma^A_B \nabla_c^\top V^B \stackrel\Sigma= \nabla_c^\top V^A+N^A(\nabla_c^\top N^B) V_B\, . 
$$
Corollary~\ref{gradI} combined with 
Equation~\nn{LAB}
implies that $\nabla_c^\top N^B = L_c^B$ and this completes the proof.
\end{proof}

\section{Extrinsic conformal Laplacian powers}
\label{invops}

An important component in our calculus  is the construction of extrinsically coupled invariant differential operators.
The key notion here are tangential operators
as defined in~\cite{GW}.

\begin{definition}\label{tangential}
Let $\sigma$ be a defining density and ${\mathcal O}$ be a smooth map on tractor bundle section spaces $\Gamma(\ct^\Phi M[w])\to
\Gamma(\ct^{\Phi'} M[w'])$. Then if
\begin{equation}\label{Os=sO}
{\mathcal O}\circ \sigma=\sigma \circ {\mathcal O'}
\end{equation}
where here $\sigma$ denotes the multiplicative operator sending $\Gamma(\ct^{\Phi} M[w])\to
\Gamma(\ct^{\Phi} M[w+1])$ (for any $\Phi$)
and ${\mathcal O'}$ is any smooth section map $\Gamma(\ct^\Phi M[w+1])\to
\Gamma(\ct^{\Phi'} M[w'+1])$, we call the operator ${\mathcal O}$ {\it tangential}.
\end{definition}

The above definition extends to vector bundles where multiplication of sections by a defining density $\sigma$ is well-defined.

\begin{example}
The map $\Gamma(\wedge^\bullet M[w])\to
\Gamma(\wedge^\bullet M[w+1])$, defined in a choice of scale by 
$$
\omega \mapsto \sigma d\omega -w\,  \varepsilon(n) \omega
$$
with as usual $n=\nabla\sigma$, is tangential.
\end{example}

\begin{remark}
Tangential operators are of particular interest because
we may define
$$\overline{\mathcal O}:\Gamma\big(\ct^\Phi M[w]\big|_\Sigma\big)\to
\Gamma\big(\ct^{\Phi'} M[w']\big|_\Sigma\big)\quad \mbox{ by }\quad
\overline {\mathcal O} \, \bar T := ({\mathcal O} T)\big|_\Sigma\, ,
$$
where $T\in  \Gamma(\ct^\Phi M[w])$
and $\bar T=T|_\Sigma$.
%
%
%
%
\end{remark}

In~\cite{GW}, it is proved that for {\it any} defining density  the operator 
$$
{\mathcal P}_k^\sigma:\Gamma\Big(\ct^\Phi M\Big[\frac{k-d+1}{2}\Big]\Big)\rightarrow \Gamma\Big(\ct^\Phi M\Big[\frac{-k-d+1}{2}\Big]\Big)\, ,\quad k\in {\mathbb Z}_{\geq1}\, ,
$$
defined by
\begin{equation}\label{holP}
{\mathcal P}^\sigma_k:=\Big(\!-\frac{1}{I_\sigma^2}\, I_\sigma .D\Big)^k,
\end{equation}
is tangential. Moreover, for AE structures it is shown that the above gives a holographic formula for the conformally invariant Laplacian powers of~\cite{GJMS}. In~\cite[Section 7.1]{GW15} it is shown that, by taking $\sigma$ to be a conformal unit defining density, the above construction  gives extrinsically coupled analogues~$\Pop_k$ of conformally invariant Laplacian powers determined by the conformal embedding $\Sigma\hookrightarrow M$.
An interesting feature is that for $k$ odd, the construction naturally produces a leading term in which the trace-free second fundamental form patly replaces the role of the inverse metric.
Here we will exploit our conformal calculus to compute explicit formul\ae\ for~$\Pop_k:=\overline {{\mathcal P}_k^\sigma}$ with~$k=2,3$.

\begin{proposition}\label{GJMS23}
Acting on tractors of weight~$\frac{3-d}{2}$ and~$\frac{4-d}2$, respectively, 
\begin{equation*}
\begin{array}{lcll}
\Pop_2&\stackrel\Sigma=&\Delta^{\!\top}+\frac{3-d}2\Big[\bar \J-\frac{K}{2(d-2)}\Big]
\, ,& d\geq 3\, ,\\[3mm]
\Pop_3&\stackrel\Sigma=&-8\Big[
\IIo^{ab}\nabla_a^\top\nabla_b^\top+\big(\nablab.\IIo^b-n^a\RR^{\!\!\!\bm \sharp}_{a}{\!\!}^b\big)\nabla_b^\top
-\frac12 \, n^a(\nabla^b\RR_{ab}^{\!\!\!\bm \sharp}	)\\[1mm]&&\quad
-\frac12 \, \frac{2-\frac d2}{d-3}\, \Big(\nablab.\nablab.\IIo-(d-4)\IIo^{ab}
\bar\Rho_{ab}
+(d-2)\IIo^{ab}{\mathcal F}_{ab}\big)
\Big)\Big] 
\, ,&d\geq4\, .
\end{array}
\end{equation*}
\end{proposition}

\begin{proof}
The result for~$\Pop_2$ was proven in Lemma~\ref{TANGD}.
For~$\Pop_3$, we initially assume only~$d\geq 3$ and now compute along~$\Sigma$:
\begin{eqnarray*}
\frac14 \, \Pop_3&=&I_{\sigma}.\hD\  I_{\sigma}.D\  I_{\sigma}.\hD\\
& =&
\Big(\nabla_n+\frac d2 H\Big)\ 
\Big(-\sigma\big[\Delta+\big(1-\frac d2\big)\J\big]\Big)\
\Big(\nabla_n+(2-\frac d2)\rho-\frac{\sigma}2\big[\Delta+\big(2-\frac d2\big) \J\big]\Big)\\[2mm]
&=&[\nabla_n,\Delta]\, +\, 2H\Delta\, +\, \big(\J+(d-4)(\nabla_n\rho)\big)\nabla_n\, -\, (d-4)(\nablab^a H)\nabla_a^\top\\[1mm]
&+&
\frac{d-4}{2}\, \big(\!-\!(\bar\Delta H)+(\nabla_n^2\rho)+(d-2)H(\nabla_n\rho)
-(\nabla_n\J)-H\J\big)\, .
\end{eqnarray*}
To obtain the last line we used the operator identities~$[\nabla_n,\sigma]\stackrel\Sigma=1$ and
$[\Delta,\si]\stackrel\Sigma=2[\nabla_n+\frac d2 H]$ (these follow from  Lemma~\ref{HII}) 
and then Lemma~\nn{Laplaces} to handle the Laplace operator acting on densities along~$\Sigma$.
To expedite the following computations we introduce the notation 
$$
\RR_{ab}^\sharp:=[\nabla_a,\nabla_b]\, ,
$$
for the  operator given by the commutator of  connections acting on mixed tensor-tractor quantities. So in particular, for any $\Gamma(TM)$-valued operator $v^c$, we have the
operator identity 
$$\RR_{ab}^\sharp \circ v^c = \RR_{ab}^{\!\!\! \bm\sharp} \circ v^c+ R_{ab}{}^c{}_d v^d\, .
$$
We now focus on the first two operators in the last line of the first display of the proof above. 
\begin{eqnarray*}
[\nabla_n,\Delta]+2H\Delta&=&\phantom{-\!}\{n^a\RR_{ab}^{\sharp}-(\nabla_b n_a)\nabla^a,\nabla^b\}+2H\Delta\\
&\stackrel\Sigma=&\phantom{-}
2n^a\RR^{\!\!\!\bm \sharp}_{a}{\!\!}^b\, \nabla_b^\top-n^b\Ric_{ba}\nabla^a+(n^a\nabla^b\RR_{ab}^{\!\!\!\bm\sharp})\\&&-2(\II^{ab}+n^an^bH)\nabla_a\nabla_b+2H\Delta-(\Delta n_a)\nabla^a
\\[2mm]
&=&-2\IIo^{ab}\nabla_a^\top\nabla_b^\top+
\big(2n^a\RR^{\!\!\!\bm \sharp}_{a}{\!\!}^b-(d-2)(n^a\Rho_a^b)^\top\big)\nabla_b^\top\\&&
-\big((d-2)\Rho(n,n)+\J+2\IIo_{ab}\IIo^{ab}\big)\nabla_n+(n^a\nabla^b\RR_{ab}^{\!\!\! \bm\sharp})
\\
&&
-(\nabla^\top_b[\II^{ab}+n^a n^b H])\nabla_a
-n_b(\nabla_n\nabla^b n^a)\nabla_a
\\[2mm]
&=&
-2\IIo^{ab}\nabla_a^\top\nabla_b^\top+
\big(2n^a\RR^{\!\!\!\bm \sharp}_{a}{\!\!}^b-(d-2)(n^a\Rho_a^b)^\top-\nablab.\IIo^b-\nablab^b H\big)\nabla_b^\top\\&&
-\big((d-2)\Rho(n,n)+\J+\IIo_{ab}\IIo^{ab}\big)\nabla_n+(n^a\nabla^b\RR_{ab}^\sharp)
\\&&
-n_b\big(\frac12\nabla^b \nabla^a n^2+R_{c}{}^{ba}{}_dn^c n^d-(\nabla_b n^c )(\nabla_c n^a)\big)\nabla_a
\end{eqnarray*}
In the first line, note that $\{\cdot,\cdot\}$ denotes the operator anticommutator while the second and third lines employed Lemma~\ref{HII}.
Using the antisymmetry of the Riemann tensor in its first two slots and the conformal unit defining density property~\nn{nid}, since $d\geq3$, 
the
very last line of the above display becomes
$$
\big(\nabla_n\nabla^a(\rho\si)\big)\nabla_a+(\nabla_n n^c)(\nabla_c n^a)\nabla_a\stackrel\Sigma=-(\nablab^a H)\nabla_a^\top+2(\nabla_n\rho)\nabla_n\, .
$$
Using this and the traced-Codazzi--Mainardi Equation~\nn{Mainarditrace} 
we obtain
\begin{equation*}
\begin{split}
[\nabla_n,\Delta]+2H\Delta\stackrel\Sigma=
&-2\IIo^{ab}\nabla_a^\top\nabla_b^\top +
\big(2n^a\RR^{\!\!\!\bm \sharp}_{a}{\!\!}^b-2\nablab.\IIo^b+(d-4)\nablab^b H\big)\nabla_b^\top
\\&
-\big(\J+(d-4)\nabla_n\rho\big)\nabla_n+(n^a\nabla^b\RR_{ab}^{\!\!\!\bm\sharp})\, .
\end{split}
\end{equation*}
Putting the above identity together with the first display of this proof we have
\begin{eqnarray*}
\frac14\,  \Pop_3
&=&-\, 2\IIo^{ab}\nabla_a^\top\nabla_b^\top
-2\big(\nablab.\IIo^b-n^a\RR^{\!\!\!\bm \sharp}_{a}{\!\!}^b\big)\nabla_b^\top
+(n^a\nabla^b\RR_{ab}^\sharp)
\\&&
+\, \frac{d-4}{2}\big(-(\bar\Delta H)+(\nabla_n^2\rho)+(d-2)H(\nabla_n\rho)
-(\nabla_n\J)-H\J\big)\, .
\nonumber
\end{eqnarray*}
The term $\nabla_n^2 \rho$ involves four normal derivatives of the conformal unit defining density~$\si$ so is only determined by the hypersurface embedding when $d\geq 4$ which we henceforth assume.  Using Equation~\nn{boxH}
to handle $\bar \Delta H$ and Lemmas~\ref{lineWillmore},~\ref{rhonnn}
and~\ref{Einstein} for $\nabla_n\rho,\nabla_n^2 \rho$, as well as the Fialkow--Gau\ss\ Equation~\nn{Gausstrace}, we obtain  the quoted result for $\Pop_3$.
\end{proof}

\begin{remark}\label{conjecture}
Note that $\Pop_2$ is a Laplace-type operator in the usual sense. On the other hand,  viewing the trace-free second fundamental form as a proxy for the inverse metric, the leading term of $\Pop_3$ is an ``extrinsic Laplacian''.
For $k=2,3$, the explicit formul\ae \ above for $\Pop_k$ have a pole at $d=k$. Hence, a dimensional continuation argument along the lines given in Remark~\ref{dimreg}  implies that the residue
of these poles is separately conformally invariant in dimension $d=k$. For $\Pop_2$, this quantity vanishes but in dimension $d=3$, a computation similar to that given in the remark, shows that the residue is precisely the 
obstruction density. It is natural to conjecture that this property will persist for higher dimensional extrinsic conformal Laplacian powers and thus provide an alternate method to compute obstruction densities (at least modulo conformally invariant densities with lower order leading derivative structure).
\end{remark}

The operator $\Pop_2$ is seen to be the intrinsic tractor-coupled Yamabe operator 
plus an invariant extrinsic term (proportional to the rigidity density).
Since it is the result of a lengthy computation, it is worthwhile demonstrating conformal invariance of~$\Pop_3$. This is done in the following lemma and proposition:

\begin{lemma}
Let $F_{ab}\in\Gamma(\Lambda^2 M)$ and view $F_{ab}$ as a weight zero operator on tractors, acting by multiplication. Then $\nabla^a\circ F_{ab}+F_{ab}\circ \nabla^a$ is an invariant operator on weight $2-\frac d2$ tractors  mapping $\Gamma(\ct^\Phi M[2-\frac d2])\to \Gamma(T^*M\otimes\ct^\Phi M[-\frac d2])$.
\end{lemma}

\begin{proof}
Writing $\nabla^a\circ F_{ab}+F_{ab}\circ \nabla^a=2F_{ab}\circ \nabla^a + (\nabla^a F_{ab})$, one only needs to compute each term for $\Omega^2 g\in \cc$ acting on $T\in\Gamma(\ct^\Phi M[2-\frac d2])$. For the first we have $\big(2F_{ab} \nabla^b T\big)|_{\Omega^2 g}=2 F_{ab} (\nabla^b+[2-\frac d2]\Upsilon^b) T$ 
where $\Upsilon=\Omega^{-1}\! \!{\bm d}\! \Omega$ and the right hand side is  given for~$g\in\cc$. For the term $(\nabla^a F_{ab})$, one must compute the transformation of the Levi-Civita connection acting on a two form. It is not difficult to verify that this exactly cancels the inhomogeneous term produced by the first term.
\end{proof}
\noindent
This lemma implies that the operator
$n^a\big[\RR^{\!\!\!\bm \sharp}_{a}{\!\!}^b\circ \nabla_b^\top
+\nabla_b^\top\circ\RR_{a}^{\!\!\! \bm\sharp}{\!\!}^b\big]$ is conformally invariant.
The following proposition  expresses $\Pop_3$ as a sum of this operator, an invariant extrinsic term $L$ and a manifestly invariant tractor operator:
\begin{proposition}
When $d\geq 4$,
$$
\Pop_3\stackrel\Sigma=-8 \, \Big[L^{AB}+\frac{ X^A\! X^BL}{d-3}\Big] {\hD^T}_A{\hD^T}_B +
4\, n^a\big[\RR^{\!\!\!\bm \sharp}_{a}{\!\!}^b\circ \nabla_b^\top
+\nabla_b^\top\circ\RR_{a}^{\!\!\! \bm\sharp}{\!\!}^b\big]
-4\, (d-4) L
\, .$$
\end{proposition}
\begin{proof}
We use the following:
(i) Remark~\ref{DTYam} to define~$L^{AB}{\hD^T}_A$ at interior Yamabe weight, 
(ii)~the result for the tractor second fundamental form in~\nn{LAB}, (iii) the canonical tangential Thomas D-operator~\nn{DTA}
and (iv) the tractor connection as given in~\nn{trconn},
 to compute an operator identity on weight~$2-\frac d2$ tractors. This gives
\begin{equation*}\begin{split}
L^{AB}&{\hD^T}_A{\hD^T }_B\ =\  \IIo^{ab} \left[ \nabla^\top_a \Big(\nabla^\top_b+\Big[2 - \frac d2\Big]n_bH\Big)
+\Big[2-\frac d2\Big]\Rho_{ab}
\right]
-\frac{\bar\nabla.\IIo^a}{d-2}\Big[2-d\Big]\nabla^\top_a \\[1mm]
&\qquad\quad\ \,  +\frac{\bar\nabla.\bar\nabla.\IIo+(d-2)\bar\Rho_{ab}\IIo^{ab}}{(d-2)(d-3)}\, \Big[1-\frac d2\Big]\Big[2-\frac d2\Big] \\[1mm]
=\ &\ 
\IIo^{ab}  \nabla^\top_a\nabla^\top_b+\nablab.\IIo^a \nabla^\top_a+\Big[2-\frac d2\Big]\Big(H\IIo^{ab}\IIo_{ab}+\IIo^{ab}\Rho_{ab}
-\frac12\frac{\bar\nabla.\bar\nabla.\IIo+(d-2)\bar\Rho_{ab}\IIo^{ab}}{(d-3)}\Big)\\
=\ &\
\IIo^{ab}  \nabla^\top_a\nabla^\top_b+\nablab.\IIo^a \nabla^\top_a-\frac12\, \frac{2-\frac d2}{d-3}\, \Big(
\bar\nabla.\bar\nabla.\IIo-(d-4)\bar\Rho_{ab}\IIo^{ab}-2(d-3)\IIo^{ab}{\mathcal F}_{ab}\Big)\, .
\end{split}
\end{equation*}
Noting that 
$n^a\RR^{\!\!\!\bm \sharp}_{a}{\!\!}^b\nabla_b^\top
+\frac12n^a(\nabla^b\RR_{ab}^{\!\!\! \bm\sharp})=\frac12\, n^a\big[\RR^{\!\!\!\bm \sharp}_{a}{\!\!}^b\circ \nabla_b^\top
+\nabla_b^\top\circ\RR_{a}^{\!\!\! \bm\sharp}{\!\!}^b\big]$
and comparing with the formula for $\Pop_3$ in Proposition~\ref{GJMS23} completes the proof.
\end{proof}

Proposition~\ref{GJMS23} gives a compact formula for the extrinsic Laplacian  appearing above when~$d=4$, thus the remaining difficulty in computing the~$d=4$ obstruction density in curved ambient spaces is calculating two normal derivatives of the canonical extension~$K_{\rm ext}$ of the rigidity density~$\IIo_{ab}\IIo^{ab}$. Since one normal derivative of the canonically extended trace-free second fundamental form is closely related to the Fialkow tensor, this boils down to computing one normal derivative of the corresponding  extension of the Fialkow tensor. This is the next natural example in the general program of proliferating natural invariants of
the conformal hypersurface structure~$(M,\cc,\Sigma)$ discussed in Section~\ref{invts}. That computation has been  performed in~\cite{GGHW}.

\section{Functionals for critical conformal hypersurface invariants}
\label{critfs}
 
We now consider the construction of critical weight Lagrangian
densities along the conformal hypersurface and thus seek Riemannian
hypersurface invariants, of weight $-d+1$, that yield conformally
invariant integrals. While it is straightforward to construct examples
(see~\cite{Guo}), the most interesting cases give action functionals
that, with respect to variation of the hypersurface embedding, yield
Euler--Lagrange equations with a linear leading
term. For hypersurfaces embedded in 5~dimensional Euclidean space,
such a functional has been constructed~\cite{Guven} by writing down a linear combination of all  possible integrated Riemannian hypersurface invariants and then fixing coefficients by demanding invariance under rigid conformal motions (see also~\cite{YuriThesis}).  These functionals are also considerable interest since they may appear as contributions to extrinsically coupled renormalized volume anomalies~\cite{GW16,Gra}.
   Constructing integrated  conformal hypersurface invariants 
  with leading derivative term quadratic in curvatures
  is rather  difficult, but a
resolution is provided via the extrinsic conformal Laplacians powers~$\Pop_{k}$
 described in Section~\ref{invops} and encapsulated by  Theorem~\ref{kin}.

\begin{proof}[Proof of Theorem~\ref{kin}]
Recall that conformal densities of weight~$-n$ 
may integrated on conformal~$n$-manifolds.
Theorem~7.1 
of~\cite{GW15} establishes that the operator $\Pop_{d-1}$ is 
determined naturally by~$(M,\cc,\Sigma)$
and thus that $N^A \Pop_{d-1} N_A$ is a weight $-d+1={\rm dim}(\Sigma)$
density along~$\Sigma$,  which yields the first statement of the Theorem.

Let $N^A$ be any  smooth extension to $M$ of the normal tractor. For $d$ odd, Theorem~7.1 of~\cite{GW15} also ensures that the operator $\Pop_{d-1}$ has non-zero leading term proportional to~$\big(\Delta^{\!\!\top}\big)^{\frac{d-1} 2}$. We will  
show below, that when $N^A \Pop_{d-1}N_A$ is integrated over the hypersurface~$\Sigma$, 
this leading term of $\Pop_{d-1}$ contributes a term of the form $\IIo^{ab}\bar\Delta^{\frac{d-3}2}\IIo_{ab}$
to the integrand. In particular this involves  $d-3$ derivatives and is quadratic in the second fundamental form. 
We will show that the lower order terms of $\Pop_{d-1}$ cannot 
contribute terms of this order to the integral $\int_\Sigma
N^A \Pop_{d-1} N_A$.
To see this, 
firstly note that  from Proposition~7.3 of~\cite{GW15} we have  
$$\Pop_{d-1}={\mathcal G}^b\circ \nabla^\top_b\, ,$$ for some smooth operator ${\mathcal G}^b$.
Hence
$$N^A \Pop_{d-1} N_A\stackrel\Sigma=[N^A,{\mathcal G}^b]\circ \nabla^\top_b N_A\, ,$$
because $N^A\nabla^\top N_A\stackrel \Sigma=0$. But, because the operator $\nabla^\top$ is tangential, we may use Corollary~\ref{gradI} to see that $\nabla_b^\top N_A^{\phantom{\top}}$ is linear in curvatures.
The operator~${\mathcal G}^b$ can only fail to commute
with~$N^A$ when a $\nabla^\top$, in the expression for~${\mathcal G}^b$, hits $N^A$ 
and produces a second curvature.
Apart from its leading derivative term, the operator  
${\mathcal G}^b$ is necessarily at least linear in curvatures. Hence, only the leading derivative term of $\Pop_{d-1}$ yields a term quadratic in cruvatures, as required.

Thus we can now focus on the leading term of $\Pop_{d-1}$ in the
density~$N^A \Pop_{{d-1}} N_A$ which can be rewritten as
  $N^A\nabla^b{}^\top\big(\Delta^{\!\!\top}\big)^{\frac{d-3} 2}\, \nabla_b^\top N_A$, because reordering derivatives  yields subleading terms involving curvatures. 
 Here, again up to subleading curvature terms, the operator $\nabla^\top$ equals $\bar\nabla$ twisted by the ambient tractor connection. Thus, discarding a divergence because $\Sigma$ is closed, the functional $\int_\Sigma N^A \Pop_{d-1} N_A$ has leading term proportional to 
 $$
 \int_\Sigma (\nabla^b{}^\top \!N^A)\, 
 \big(\Delta^{\!\!\top}\big)^{\frac{d-3} 2}\, \nabla_b^\top N_A^{\phantom{\top}}\, .
 $$
We may use Corollary~\ref{gradI} again to see that $[\nabla_b^\top N_A^{\phantom{\top}}]$  has the form
$$
[\nabla_b^\top N^A] \stackrel\Sigma=
\begin{pmatrix}
\, 0
\\[1mm]
-\IIo_b^a
\\[.5mm]\, \star
\end{pmatrix}\, .
$$
It follows that the leading term of the  functional, as claimed,   is a non-zero multiple of
$$
\int_\Sigma \IIo^{ab}\bar\Delta^{\frac{d-3}2}\IIo_{ab}\, .
 $$
 It is not difficult to check that when varying an embedding of a functional
$\int_\Sigma  \IIo^{ab} K_{ab}$, the contribution to the Euler-Lagrange equation from the variation of the (explicit) trace-free second fundamental form is
$
\nablab^a \nablab^b K_{(ab)\circ}
$,
where $K_{(ab)\circ}$ denotes the trace-free symmetric part of the tensor $K$. Since varying the measure or the operator $\bar\Delta^{\frac{d-3}2}$ necessarily leads to contributions quadratic in $\IIo$, it follows that  the functional 
in the above display contributes only 
$
2\nablab^a\nablab^b \bar\Delta^{\frac{d-3}2}\IIo_{ab}
$,
at linear order in $\IIo$, to the Euler--Lagrange equation.
Employing the  identity~\nn{boxH} we thus obtain the Euler--Lagrange equation
$$
\bar{\Delta}^{\frac{d-1}{2}} H + \mbox{\rm  lower order terms}=0\, ,  
$$
in agreement with the result of Theorem~5.1 of~\cite{GW15}  
for the leading order contribution to the obstruction density.
\end{proof}

\begin{remark}
As we discuss in the following example, the last statement of the above theorem also holds for embedded volumes, except that the  Euler--Lagrange equation is now  quadratic in the second fundamental form as it must be  to   agree with the leading term of the corresponding obstruction density.
It seems plausible  that a similar statement holds for all 
higher, odd dimensional embedded hypersurfaces.
\end{remark}

\begin{example}\label{surfc}
A simple application of our extrinsic Laplacian formul\ae\  is to compute low dimensional examples of the 
action functional density~\nn{kf}. The easiest case is dimension $d=3$ for which we find
$$
N_A \Pop_2 N^A\stackrel\Sigma=
N_A \, \bar y^2\,  I^A=N_A \Delta^{\!\top} I^A=-\IIo_{ab}\IIo^{ab}\, .
$$
The second step above used Proposition~\ref{GJMS23}
 while the last step of this computation relied on Corollary~\ref{gradI} to evaluate $\nabla_a I^A$ as well as Equations~\nn{trconn} and~\nn{hypgrad}, respectively, for the tractor connection and the
relation between tangential and boundary Levi-Civita connections. Hence the functional
$$
\int_\Sigma N_A\,  \Pop_2 N^A=-\int_\Sigma \IIo_{ab}\IIo^{ab}=-\int_\Sigma K\, ,
$$ recovers the well-known Willmore energy~\cite{Willmore} or
(extended to Lorentzian signature) the rigid string action
of~\cite{Polyakov} which justifies calling $K$ the rigidity density.

The above functional appears in the formula for
the renormalized area of a minimal surface embedded in a hyperbolic
3-manifold~\cite{Alexakis}.
It is  interesting to note that the above functional also appears as the
log term coefficient in the asymptotic expansion for the volume
associated with a 2-brane in the AdS/CFT correpondence (and is linked
to the anomaly for boundary observables)~\cite{GrahamWitten}; the corresponding anomaly functionals for hypersurfaces of arbitrary dimensions have recently been computed in~\cite{GW16}.

In the next dimension $d =4$, the computation of $N_A \Pop_3 N^A$ is more involved, but remains simple for conformally flat structures: From Proposition~\ref{GJMS23} we have in this case that $\Pop_3\stackrel\Sigma=-8\big[\IIo^{bc}\nabla_b^\top+(\nablab_b\IIo^{bc})\big]\nabla_c^\top$. We again use Corollary~\ref{gradI}, which gives
$$
[\nabla_c^\top N^A]\stackrel\Sigma=\begin{pmatrix}
0\ \\\IIo^a_c\\[1mm]-\frac{\nablab.\IIo_c}{d-2}\ 
\end{pmatrix}\, .
$$
Since $N_{\!A}^{\phantom{o}\!} \nabla_c^\top N^A = 0$,  we only need to compute the leading double derivative term which again  requires using  Equation~\nn{trconn} for the tractor connection. This yields $$N_C \IIo^{ab}\nabla^\top_a\nabla_b^\top N^C=-\IIo^{ab}\IIo_{a}^c\IIo_{cb}^{\phantom{c}}=-\IIo^{ab}\F_{ab}=-L\, .$$ Hence, as promised, $L$ plays the {\it r\^ole} of a rigidity density for embedded volumes. 
Indeed, for conformally flat structures, it is straightforward to 
compute the embedding  variation of the functional
$$
\int_\Sigma N_A \Pop_3 N^A = 8\int_\Sigma \IIo^{ab}\IIo_{a}^c\IIo_{cb}^{\phantom{c}}=8\int_\Sigma L\, .
$$ 
(Functionals constructed from powers of $\IIo$ have been studied in~\cite{Guo}.)
The resulting Euler--Lagrange equation is ${\mathcal B}=0$ with
${\mathcal B}$ given by the conformally-flat, four-manifold,
obstruction density quoted in Proposition~\ref{4Will}. Details of this
computation and its extension to generally curved conformal structures
is presented in~\cite{GGHW}. There it is shown that, for hypersurfaces
in general 4-manifolds, the functional gradient of \nn{kf} agrees
precisely with the obstruction density.  We note that the functional
$\int_\Sigma L$ in Lorentzian signature could be of interest for a
rigid membrane theory.
\end{example}

\appendix

\section{Proof of Lemma~\ref{W3}}

\label{nablan3rho}
In this section, we employ the notations of section~\ref{comp-ext} and break the proof of Proposition~\ref{W3} into several smaller pieces. The first of these
explicates the terms ``{\it \rm LTOTs}'' of Equation~\nn{rho-step}.

\begin{lemma}
Let $\sigma$ be a conformal unit defining density, then
\begin{eqnarray}
\frac12 \, \nabla_n^3 I^2_{\sigma} +(d-3)\nabla_n^2\rho
=-\nabla_n^2 \big(\gamma^{ab} \nabla_a n_b\big)-\nabla_n\big(5\rho^2+2\J\big)+4\rho^3
+2\rho \J\, .\label{threederivs}
\end{eqnarray}
\end{lemma}

The proof of the above Lemma is, by now, elementary.
Of the terms on the right hand side of~\nn{threederivs}, only the 
the first has not been computed from previous steps in the recursion.
This is somewhat involved. Firstly, we need a lemma relating the 
ambient and hypersurface Laplacians.
\begin{lemma}\label{Laplaces}
Let~$f$ be a (smooth) extension of any function~$\bar f$ defined along~$\Sigma$. Then
$$
\bar\Delta\bar f\stackrel\Sigma=\big(\Delta-\nabla_n^2-(d-2) H\nabla_n\big) f\, .
$$
\end{lemma}

\begin{proof}
The proof is a simple (double) application of the formula~\eqref{hypgrad} relating 
ambient and hypersurface Levi-Civita connections
\begin{equation*}
\begin{split}
\bar \Delta\bar f&\stackrel\Sigma=
\gamma^{ab}(\nabla_a^\top\nabla_b^\top f + \hat n_b\II_a^c \nabla_c^\top f)\stackrel\Sigma=
(\nabla^a-n^a \nabla_n)(\nabla_a-n_a \nabla_n)f\\&
\stackrel\Sigma=\big(\Delta-\nabla_n^2+(\nabla_n n^a) \nabla_a - (\nabla_a^\top n^a) \nabla_n\big)f\, .
\end{split}
\end{equation*}
Finally,  note that~$\nabla_a^\top n^a\stackrel\Sigma=\II_a^a=(d-1) H$ and
$\nabla_n n_a = \frac12 \nabla_an^2 \stackrel\Sigma=\nabla_n(-\rho\sigma)\stackrel\Sigma=H$.
\end{proof}

This result allows us to compute a quantity required for handling the troublesome term $-\nabla_n^2 \gamma^{ab} \nabla_a n_b$ in Equation~\nn{threederivs}. 
\begin{lemma}\label{leadinglaplace}
Let $\sigma$ be a conformal unit defining density,  and~$d>2$, then
$$\gamma^{ab}\nabla_a\nabla_n^2 n_b\stackrel\Sigma=\bar\Delta H  -2(d-1)H\nabla_n \rho+ (d-1)H^3\, .$$
\end{lemma}

\begin{proof}
Again, we compute explicitly along~$\Sigma$ using the techniques developed in section~\ref{comp-ext} and, at the last step, the preceding lemma:
\begin{equation*}
\begin{split}
\gamma^{ab}\nabla_a \nabla_n^2 n_b&=
\frac12 \gamma^{ab}\nabla_a \nabla_n\nabla_b\big(1-2\rho\sigma+\O(\sigma^d)\big)\\
&=- \gamma^{ab}\nabla_a\nabla_n\big(\sigma\nabla_b \rho + \rho n_b\big)\\
&=- \gamma^{ab}\nabla_a \big(\nabla_b \rho (1-2\rho\sigma)+ \nabla_n \rho n_b + \rho \nabla_n n_b\big)\\
&=- \gamma^{ab}\nabla_a \nabla_b \rho-(d-1)H\nabla_n \rho- \gamma^{ab}\rho\nabla_a(-\sigma\nabla_b \rho  -\rho n_b)\\
&=\bar\Delta H  -2(d-1)H\nabla_n \rho+ (d-1)H^3\, .
\end{split}
\end{equation*}
\end{proof}

\begin{remark}
This result ensures that the leading term of the $d=3$ obstruction density 
coincides with the leading Laplacian term of the Willmore invariant~\nn{Wore}.
\end{remark}

To use Lemma~\ref{leadinglaplace}, we still need to commute the operators~$\nabla_n^2$ and~$\gamma^{ab}\nabla_a$. This calculation
is encoded in the following result.
\begin{lemma}\label{commutators}
Let $\sigma$ be a conformal unit defining density and~$d>2$, then
\begin{equation*}
\begin{split}
\nabla_n^2 \gamma^{ab} \nabla_a n_b 
-\gamma^{ab} \nabla_n^2 \nabla_a n_b&\stackrel\Sigma=12 H \nabla_n \rho -4H^3\, ,\\[2mm]
\gamma^{ab}\big( \nabla_n^2 \nabla_a n_b -\nabla_a \nabla_n^2 n_b\big)\ 
&\stackrel\Sigma= -(\nabla_n-H)\Ric(n,n)+2\IIo^{ab}\IIo_{ac}\IIo^c_b+3H\IIo^{ab}\IIo_{ab}\\
&\phantom{=}\  - \!(d-1)H^3-2\IIo^{ab}R_{cabd}\, n^c n^d\, .
\end{split}
\end{equation*}
\end{lemma}

\begin{proof}
Again, both these results can be obtained computing along~$\Sigma$
using the techniques developed in section~\ref{comp-ext}:
\begin{equation*}
\begin{split}
[\nabla_n^2, \gamma^{ab}] \nabla_a n_b&=  
-2 n^a (\nabla_n^2  n^b) \nabla_a n_b -2 (\nabla_n n^a) (\nabla_n n^b) \nabla_a n_b  -4n^a (\nabla_n n^b) \nabla_n \nabla_a n_b\\[1mm]
&=-\, 2 n^a (-\nabla^b\rho - n^b\nabla_n\rho+ H^2 n^b)(\II_{ab}+H n_a n_b)-2 H^2 n^a n^b  (\II_{ab}+H n_a n_b)\\
&\phantom{=\ }-4 n^a H n^b (\nabla_a \nabla_n n_b + R_{cabd}\, n^c n^d-\nabla_a n^c \nabla_c n_b)\\[1mm]
&=2H(2\nabla_n\rho  - H^2)-2H^3-4Hn^a \nabla_n^2 n_a +4 H^3\\[1mm]
&=12 H \nabla_n \rho -4H^3\, ,
\end{split}
\end{equation*}
and
\begin{equation*}
\begin{split}
\gamma^{ab}\big( &\nabla_n^2 \nabla_a n_b -\nabla_a \nabla_n^2 n_b\big)
= \gamma^{ab}\big([\nabla_n,\nabla_a]\nabla_n n_b + \nabla_n\big(R_{cabd}\, n^c n^d-(\nabla_a n_c) \nabla^c n_b\big)\big)\\[2mm]
&=\gamma^{ab} \big(R_{dabc}n^d\nabla_n n^c -(\nabla_a n_c) \nabla^c \nabla_n n_b\big)\\
&\phantom{=\ }-\nabla_n\Ric(n,n)+2n^a (\nabla_n n^b) R_{cabd}\, n^c n^d
-2\gamma^{ab} (\nabla_a n_c) \nabla_n\nabla_b n^c\\[2mm]
&=-H\Ric(n,n)-3\II^{ab} \nabla_a \nabla_n n_b-\nabla_n\Ric(n,n)
-2\II^{ab}\big(R_{cabd}\, n^c n^d-\nabla_a n_c \nabla^c n_b\big)\\[2mm]
&=-(\nabla_n+H)\Ric(n,n)-3H\II^{ab}\II_{ab}-2\II^{ab}R_{cabd}\, n^c n^d+2\II^{ab}\II_{ab}\II^c_b\\[2mm]
&=-(\nabla_n-H)\Ric(n,n)+2\IIo^{ab}\IIo_{ac}\IIo^c_b+3H\IIo^{ab}\IIo_{ab}
-(d-1)H^3-2\IIo^{ab}R_{cabd}\, n^c n^d
\, .
\end{split}
\end{equation*}
\end{proof}

Orchestrating Lemmas~\ref{threederivs},~\ref{leadinglaplace} and~\ref{commutators}
plus the results  of section~\ref{comp-ext} for the previous steps of the recursion involving~$\rho$ and~$\nabla_n\rho$ along~$\Sigma$, gives immediately our first formula for the~$d=3$ obstruction density and~$\nabla_n^2 \rho$.

\begin{lemma}\label{rhonnn}
Let $\sigma$ be a conformal unit defining density, then if~$d>2$, 
\begin{equation*}
\begin{split}
\frac12 \, \nabla_n^3 I^2_{ \sigma}&+(d-3)\nabla_n^2\rho=
-\bar\Delta H  -H\IIo^{ab}\IIo_{ab} -2\IIo^{ab}\IIo_{ac}\IIo^c_b
\\&+\nabla_n G(n,n)+(d-3) (\nabla_n+2H)\J
+2\IIo^{ab}R_{cabd}\, n^c n^d+H\Ric(n,n)\, .
\end{split}
\end{equation*}
\end{lemma}

To complete the proof of Proposition~\ref{W3} we need to (i) express the ambient
Riemann tensor in terms of its Weyl and Schouten tensor constituents, (ii) trade the Laplacian of mean curvature for the  second fundamental form divergence using the hypersurface identity~\nn{boxH} and (iii) rewrite the normal derivative of the normal components of the 
ambient Einstein tensor~$\nabla_n G(n,n)$ in terms of hypersurface quantities.
Only step (iii) is non-trivial, it relies on one more Lemma.
\begin{lemma}\label{Einstein}
$$
\nabla_n G(n,n)\stackrel\Sigma=
-\nablab^a(\Ric_{ab}\hat n^b)^\top+\IIo^{ab} \Ric_{ab}-(d-2) H \Ric(n,n)\, .
$$
\end{lemma}

\begin{proof}
This computation relies on the algebraic Bianchi identity for the ambient Riemann tensor:
\begin{equation*}
\begin{split}
\nablab^a&\big((\Ric_{ab} n^b)^\top\big)
=\gamma^{ab}\nabla_a\big(R_{dcb}{}^c\hat n^d-n_bR_{dce}{}^cn^dn^e\big)\\
&=\II^{ab}\Ric(a,b)-\gamma^{ab}n^d \big(\nabla_d R_{cab}{}^{c}+\nabla_c R_{adb}{}^{c}\big)-(d-1)H\Ric(n,n)\\
&=\II^{ab}\Ric(a,b)+\frac12\nabla_n R-\nabla_n \Ric(n,n)+\nabla_n(n^an^b) \Ric_{ab}
-(d-1)H\Ric(n,n)\\
&=-\nabla_n G(n,n) +\IIo^{ab} \Ric_{ab}-(d-2) H \Ric(n,n)\, . 
\end{split}
\end{equation*}
\end{proof}

\newcommand{\msn}[2]{\href{http://www.ams.org/mathscinet-getitem?mr=#1}{#2}}
\newcommand{\hepth}[1]{\href{http://arxiv.org/abs/hep-th/#1}{arXiv:hep-th/#1}}
\newcommand{\maths}[1]{\href{http://arxiv.org/abs/math/#1}{arXiv:math/#1}}
\newcommand{\mathph}[1]{\href{http://lanl.arxiv.org/abs/math-ph/#1}{arXiv:math-ph/#1}}
\newcommand{\arxiv}[1]{\href{http://lanl.arxiv.org/abs/#1}{arXiv:#1}}

\end{document}